\documentclass[reqno]{amsart}

\pdfoutput=1

\usepackage{amsmath,amsthm, amscd, amssymb, amsfonts, mathrsfs}
\usepackage[all]{xy}
\usepackage{color}

\usepackage[T1]{fontenc}
\usepackage{array,multirow}
 \usepackage{tikz}
\usetikzlibrary{patterns}
\usetikzlibrary{arrows}

\usepackage{multicol}
\usepackage{enumitem}

\usepackage{cite}

\usepackage[plainpages=false,pdfpagelabels,
hypertexnames=false]{hyperref}

\allowdisplaybreaks

\newcommand{\cO}{\mathcal{O}}
\newcommand{\cS}{\mathcal{S}}

\newcommand{\FK}{\mathcal{FK}}

\newcommand{\Sn}{{\mathbb S}}

\newcommand{\xtop}{x_{top}}

\newcommand{\xij}[1]{x_{(#1)}}
\newcommand{\yij}[1]{y_{(#1)}}

\newcommand{\DSn}{\D(\Sn_3)}
\newcommand{\oV}{\overline{V}}

\newcommand{\ot}{{\otimes}}
\newcommand{\ku}{\Bbbk}

\newcommand\fA{\mathsf{A}}
\newcommand\fB{\mathsf{B}}
\newcommand\fC{\mathsf{C}}

\newcommand\fE{\mathsf{E}}

\newcommand\fG{\mathsf{G}}
\newcommand\fH{\mathsf{H}}

\newcommand\fInd{\mathsf{Ind}}

\newcommand\fL{\mathsf{L}}
\newcommand\fM{\mathsf{M}}
\newcommand\fN{\mathsf{N}}

\newcommand\fP{\mathsf{P}}

\newcommand\fR{\mathsf{R}}
\newcommand\fS{\mathsf{S}}
\newcommand\fT{\mathsf{T}}

\newcommand\fW{\mathsf{W}}
\newcommand\fX{\mathsf{X}}

\newcommand\fe{\mathsf{e}}

\newcommand\fh{\mathsf{h}}

\newcommand\fm{\mathsf{m}}
\newcommand\fn{\mathsf{n}}
\newcommand\fo{\mathsf{o}}
\newcommand\fp{\mathsf{p}}

\newcommand\fs{\mathsf{s}}
\newcommand\ft{\mathsf{t}}
\newcommand\fu{\mathsf{u}}

\newcommand{\Z}{{\mathbb Z}}

\newcommand{\D}{\mathcal{D}}
\newcommand{\BV}{{\mathfrak B}}

\newcommand{\Ext}{\operatorname{Ext}}

\newcommand{\ch}{\operatorname{ch}}
\newcommand{\chgr}{\operatorname{ch}^{\bullet}}

\newcommand{\Lsp}{\Lambda_{sp}}

\newcommand\sgn{\operatorname{sgn}}

\newcommand\soco{\operatorname{soc}}

\newcommand\im{\operatorname{im}}

\newcommand\mm[1]{\boldsymbol{|#1\rangle}}
\newcommand\e{\varepsilon}

\theoremstyle{plain}

\newtheorem{lema}{Lemma}[section]
\newtheorem{theorem}[lema]{Theorem}

\newtheorem{prop}[lema]{Proposition}

\newtheorem{question-app}{Question}

\theoremstyle{definition}
\newtheorem{definition}[lema]{Definition}

\theoremstyle{remark}
\newtheorem{obs}[lema]{Remark}
\newtheorem{rmk}[lema]{Remarks}

\begin{document}

\title[]{On the representation theory of the Drinfeld double of the Fomin-Kirillov algebra $\FK_3$}

\author[B. Pogorelsky and C. Vay]{Barbara Pogorelsky and Cristian Vay}

\address{Instituto de Matem\'atica, Universidade Federal do Rio Grande do Sul, Av. Bento Goncalves 9500, Porto Alegre, RS, 91509-900, Brazil} \email{barbara.pogorelsky@ufrgs.br}

\address{Universidad Nacional de C\'ordoba, Facultad de Matem\'atica, Astronom\'ia, F\'isica y Computaci\'on, CIEM--CONICET, C\'ordoba, Rep\'ublica Argentina} \email{vay@famaf.unc.edu.ar}

\thanks{
\newline C. V. was partially supported by CONICET, Secyt (UNC), FONCyT PICT 2016-3957, Programa de Cooperaci\'on MINCyT-FWO, MathAmSud project GR2HOPF and ESCALA Docente AUGM}

\begin{abstract}
Let $\D$ be the Drinfeld double of $\FK_3\#\ku\Sn_3$. The simple $\D$-modules were described in \cite{PV2}. In the present work, we describe the indecomposable summands of the tensor 
products between them. We classify the extensions of the simple modules and show that $\D$ is of wild representation type. We also investigate the projective modules and 
their tensor products.
\end{abstract}

\maketitle

\section{Introduction}

An important property  of the category of modules over a Hopf algebra is that it is a tensor category. Several works address the study of the tensor structure for various families of Hopf 
algebras: the small quantum group $u_q(\mathfrak{sl_2})$ \cite{MR2774620,GaiSemTip06,MR1284788}, the (generalized) Taft algebras \cite{MR1243707,MR1987337,MR3148512,MR3077243} and their Drinfeld 
doubles \cite{MR3655701,MR3260906,MR2441478,MR2184820}, the Drinfeld doubles of finite groups \cite{MR1367852}, the non-semisimple Hopf algebras of low dimension \cite{MR1969453} and 
the pointed Hopf algebras over $\ku\Sn_3$ \cite{MR2681259} (these are liftings of the Fomin-Kirillov algebra $\FK_3$). 

In particular, the small quantum group $u_q(\mathfrak{sl_2})$ is a quotient of the Drinfeld double of a Taft algebra by central group-like elements. Thus, their representation theory have in common significant features: (1) the simple modules are parametrized by the simple modules over the corresponding corradical and (2) the tensor product of two simple modules decomposes into the direct sum of simple and projective modules.

These features are generalization of well-known results in Lie theory. Indeed, the simple modules over a semisimple Lie algebra are parametrized by the weights of the Cartan subalgebra, 
while the tensor products of simple modules are described by the Clebsch-Gordon formula. Moreover, (1) holds for Drinfeld doubles of bosonizations of finite-dimensional Nichols algebras over finite-dimensional Hopf algebras, see for instance \cite{arXiv:1802.00316,arXiv:1705.08024,MR2279242,PV2,vay-mca17}. Notice 
that a Taft algebra can be presented as a bosonization of the quantum line $\ku\langle x\mid x^n=0\rangle$ over the cyclic group of order $n$, the first example of a finite-dimensional 
Nichols algebra. 

A valuable consequence of (2) is, roughly speaking, that the simple modules generate a fusion subcategory in a quotient category. The motivating question for our work was: will (2) also hold for other Drinfeld doubles? 

In the present work, we address this question for the Drinfeld double $\D$ of $\FK_3\#\ku\Sn_3$, {\it i.e.} the bosonization of the Fomin-Kirillov algebra $\FK_3$ over the symmetric group $\Sn_3$. We point out that $\FK_3$
is the first example of a finite-dimensional Nichols algebra over a non-abelian group \cite{MR1800714,MR1667680}. Next, we summarize our main results.

\

The simple $\D$-modules are parametrized by the simple modules over the Drinfeld double $\DSn$ of $\ku\Sn_3$ which play the role of {\it weights} in this setting. Let us denote by
$$
\Lambda=\left\{\varepsilon=(e,+),\, (e,-),\, (e,\rho),\, (\sigma,+),\, (\sigma,-),\, (\tau,0),\, (\tau,1),\, (\tau,2)\right\}.
$$
the set of weights. Table \ref{tab:weight} condenses basic information about them and explains the notation, see also \cite[\S 5.2]{PV2}. We recall that the simple $\DSn$-modules are classified by the conjugacy classes of $\Sn_3$ and the irreducible representations of the respective centralizers (this holds for any finite group $G$ not only for $\Sn_3$, see for instance \cite{MR1714540}). 

\begin{table}[h]
\begin{center}
\begin{tabular}{r|c|c|c|c|c|c|c|c|}
weight & $\varepsilon$ & $(e,-)$ & $(e,\rho)$ & $(\sigma,+)$ & $(\sigma,-)$ & \multicolumn{3}{c|}{$(\tau,i)$, $i=0,1,2$}

\\ \hline

dimension & 1 & 1 & 2 & 3 & 3 & \multicolumn{3}{c|}{2}

\\ \hline

conjugacy & \multicolumn{3}{c|}{$e\in\Sn_3$,} & \multicolumn{2}{c|}{$\sigma\in\Sn_3$,} & \multicolumn{3}{c|}{$\tau\in\Sn_3$,}

\\

class of & \multicolumn{3}{c|}{the identity element} & \multicolumn{2}{c|}{a transposition} & \multicolumn{3}{c|}{a $3$-cycle}

\\ \hline

centralizer & \multicolumn{3}{c|}{$\Sn_3$} & \multicolumn{2}{c|}{cyclic group $C_2$} & \multicolumn{3}{c|}{cyclic group $C_3$}

\\ \hline

simple  & trivial & sign & 2-dim & trivial & sign & \multicolumn{3}{c|}{$i$-th power of a}

\\

representation  & && & & & \multicolumn{3}{c|}{$3$-root of unity}

\\ \hline

\end{tabular}
\end{center}
\caption{Weights}
\label{tab:weight}
\end{table}

Let $\{\fL(\lambda)\}_{\lambda\in\Lambda}$ be the family of (non-isomorphic) simple $\D$-modules. They are characterized as follows, see {\it e.g.} \cite{PV2}. We have a triangular decomposition $\D\simeq\BV(V)\ot\D(\Sn_3)\ot\BV(\oV)$ where $\BV(V)$ and $\BV(\oV)$ are Nichols algebras isomorphic to $\FK_3$. Thus, $\fL(\lambda)$ is the unique simple $\D$-module of highest-weight $\lambda\in\Lambda$, {\it i.e.} it has a $\DSn$-submodule isomorphic to $\lambda$ such that $\BV(\oV)\cdot\lambda=0$. We give more details in \S\ref{sec:Preliminaries} and in the appendix.

\begin{figure}[h]
\begin{center}
\includegraphics{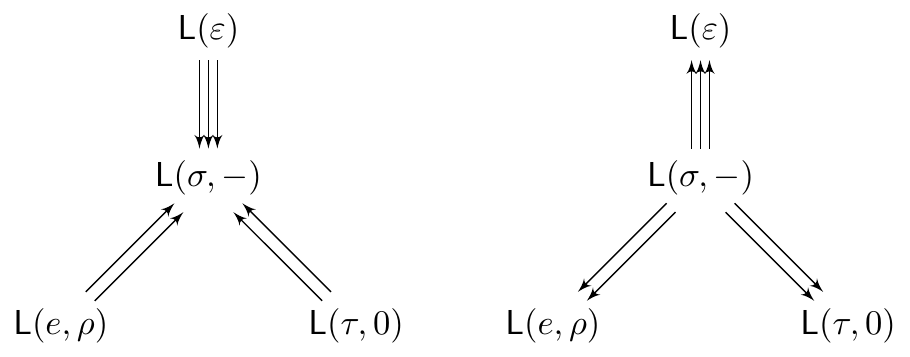}
\end{center}
\caption{Separated quiver of $\D$. The number of arrows indicates the dimension of the respective space of extensions.}
\label{fig:separated quiver}
\end{figure}

By \cite[Theorem 6]{PV2} and \cite[Corollary 17]{vay-proj}, $\fL(\lambda)$ is projective if and only if 
$$
\lambda\in\Lsp:=\{(e,-),(\sigma,+),(\tau,1),(\tau,2)\}.
$$
The remainder simple modules generate a single block of the category of $\D$-modules because they are composition factors of an indecomposable module, the Verma module of $(\sigma,-)$ \cite[Theorem 7]{PV2}. In Section \ref{sec:ext simple modules}, we compute the extensions between these simple modules and show that $\D$ is of wild representation type. We draw the separated quiver of $\D$ in Figure \ref{fig:separated quiver}.


The major effort of our work is in describing the indecomposable summands of the tensor products of simple modules.

\begin{theorem}
Let $\D$ be the Drinfeld double of $\FK_3\#\ku\Sn_3$. Given $\lambda,\mu\in\Lambda$, the indecomposable summands of the tensor product $\fL(\lambda)\ot\fL(\mu)$ are described in 
Propositions \ref{prop:1}, \ref{prop: sigma menos por sigma menos}, \ref{prop:fB}, \ref{prop:fC}, \ref{prop:remainder}, \ref{prop:simple projective} and \ref{prop:simple by projective}. 
\end{theorem}

The outcome of the above is resumed in Table \ref{tab:tensors}. We find out new indecomposable modules $\fA$, $\fB$ and $\fC$ which are not either simple or projective. We schematize them in Figures \ref{fig:fA}, \ref{fig:fB} and \ref{fig:fC}, respectively. If one of the factors is projective, the tensor product is also projective and then we can use results from \cite{vay-proj} in order to describe its direct summands. However, we do not have enough space in the table to write them except when both factors are projective. In this case, $\fInd(\lambda\cdot\mu)$ is the induced module $\D\ot_{\DSn}(\lambda\ot\mu)$ which is not necessary indecomposable. The cells under the diagonal are empty because $\D$ is quasitriangular and hence the tensor product is commutative. We do not include $\fL(\varepsilon)$ in the table because it is the unit object. 

\begin{table}[h!]
\begin{center}
\begin{tabular}{c|c|c|c|c|}
$\ot$ & $\fL(e,\rho)$ & $\fL(\tau,0)$ & $\fL(\sigma,-)$ & $\fL(\lambda)$, $\lambda\in\Lambda_{sp}$

\\ \hline

\multirow{2}{*}{$\fL(e,\rho)$} & \multirow{2}{*}{$\fL(e,-)\oplus\fB$} & $\fL(\tau,1)\oplus\fL(\tau,2)$ & \multirow{2}{*}{$\fL(\sigma,+)\oplus\fC$} &  Proposition

\\

&& $\oplus\fL(\varepsilon)$&& \ref{prop:simple by projective}

\\ \cline{1-4}

$\fL(\tau,0)$ &  &  $\fL(e,-)\oplus\fB^*$ & $\fL(\sigma,+)\oplus\fC^*$ &  

\\ \cline{1-1} \cline{3-4}

\multirow{2}{*}{$\fL(\sigma,-)$}  & \multicolumn{2}{c|}{}  & $\fL(\tau,1)\oplus\fL(\tau,2)$ & 

\\

  &  \multicolumn{2}{c|}{}  & $\oplus\fL(\varepsilon)\oplus\fA$ &

\\ \cline{1-1} \cline{4-5}

$\fL(\mu)$, $\mu\in\Lambda_{sp}$ & \multicolumn{3}{c|}{}  & $\fInd(\lambda\cdot\mu)$

\\ \hline

\end{tabular}
\end{center}
\caption{Tensor products of simple modules.}
\label{tab:tensors}
\end{table}

In conclusion, question (2) does not hold in this example. Instead, all the non-simple non-projective summands have the following in common. 
\begin{itemize}
 \item[$\circ$] They have simple head and simple socle. Moreover, these are isomorphic. 
 \item[$\circ$] Being graded, the socle and the head are concentrated in the same homogeneous components.
 \item[$\circ$] They are not either highest-weight modules or lowest-weight modules.
 \end{itemize}

The main difficult to deal with the modules $\fA$, $\fB$ and $\fC$ is that some weights have dimension greater than one, cf. Table \ref{tab:weight}. Hence the tensor product of two weights is not necessarily a weight, but it is the direct sum of various weights. These facts complicate the computations. However, the use of the following properties helps to simplify things. These properties hold in general and are not present in the above references. 

First of all, we can restrict our attention to the category of graded modules. This is because $\D$ is graded and finite-dimensional and hence simple modules, their tensor products and the indecomposable summands of the latter are graded by \cite{MR659212}. Notice that the category of graded $\D$-modules is a highest-weight category \cite{arXiv:1705.08024}. 

Let $\fN=\oplus_{i\in\Z}\fN(i)$ be a graded $\D$-module and $\chgr\fN$ its graded character, {\it i.e.} its representative in the Grothendieck ring of the category of graded $\D(\Sn_3)$-modules. Then
\begin{itemize}
 \item The graded composition factors of $\fN$ are given by $\chgr\fN$.
\end{itemize}
In fact, the graded characters of the simple modules form a $\Z[t,t^{-1}]$-basis of the Grothendieck ring of the category of graded $\D$-modules \cite[Theorem 9]{vay-proj}. However, two simple modules could have identical ungraded character as for instance $\fL(e,\rho)$ and $\fL(\tau,0)$, see Remark \ref{obs:ungraded character}.

In order to compute the indecomposable summands of $\fN$, we need to know how its composition factors are connected. For this purpose, we need to calculate the action of the space of generators 
$V$ of $\FK_3$ on a homogeneous weight $S$ of $\fN$, {\it i.e.} a simple $\DSn$-submodule of $\fN(i)$. Here, we shall use that
\begin{itemize}
 \item The action $V\ot S\longrightarrow\fN(i-1)$ is a morphism of $\DSn$-modules.
\end{itemize}
This last fact is also useful to classify the extensions of simple $\D$-modules, see Lemmas \ref{le:trivial extensions} and \ref{le:notrivial extension}.

\

The article is organized as follows. In Section \ref{sec:Preliminaries} we recall the structure of $\D$ and summarize all the notation and conventions. We study the extensions of the simple modules in Section \ref{sec:ext simple modules} and their tensor products in Section \ref{sec:tensores}. Finally, we describe the projective modules and their tensor products in Section \ref{sec:proyectivos}. In the
appendix we give the action of the generators of $\D$ on the simple modules.

\subsection*{Acknowledgments} 
We thank Nicol\'as Andruskiewitsch for suggesting us this project and for stimulating discussions. Part of this work was carried out during a visit of B. P. to the Universidad Nacional de 
C\'ordoba in the framwork of ``MathAmSud project GR2HOPF'', a visit of C. V. to the Universidade Federal do Rio Grande do Sul in the framework of ``Programa ESCALA Docente de la 
Asociasi\'on de Universidades Grupo Montevideo'', and also during a research stay of C. V. in the University of Clermont Ferrand (France) supported by CONICET. We also appreciate the thorough reading and the comments of the referee which have enriched our work.

\section{Preliminaries}\label{sec:Preliminaries}

We summarize all the information needed for our work. We follow the notation and conventions of \cite{PV2,vay-proj,vay-mca17}. Most of the properties which we will list hold for any finite dimensional Nichols algebra over a finite dimensional semisimple Hopf algebra. Nevertheless, we prefer recall them in our particular case to the benefit of the reader. The general statements could be find in {\it loc.cit.} 

\subsection{The Drinfeld double of \texorpdfstring{$\FK_3\#\ku\Sn_3$}{FK3}}\label{subsec:lo basico}

We begin by fixing the notation related to the Drinfeld double $\D=\D(\FK_3\#\ku\Sn_3)$. 

\begin{enumerate}[label=(\alph*)]
\item  We denote $\DSn$ the Drinfeld double of $\ku\Sn_3$. As an algebra, $\DSn$ is generated by the group-like elements of $\Sn_3$ and the dual elements $\delta_g$, $g\in\Sn_3$.
  \item $V=\ku\{\xij{12},\xij{13},\xij{23}\}$ is a simple $\DSn$-module via
 \begin{align*}
  g\cdot\xij{ij}=\sgn(g)x_{g(ij)g^{-1}}\quad\mbox{and}\quad \delta_g\cdot\xij{ij}=\delta_{g,(ij)}\,\xij{ij}.
 \end{align*}
  \item $\oV=\ku\{\yij{12},\yij{13},\yij{23}\}$ is the dual object of $V$ in the category of $\DSn$-modules. They are isomorphic via $\yij{ij}\mapsto\xij{ij}$.
  \smallskip
  \item The Nichols algebra $\BV(V)$ is the quotient of the tensor algebra $T(V)$ by
  \begin{align*}
  \xij{ij}^2,\quad\xij{12}\xij{13}+\xij{23}\xij{12}+\xij{13}\xij{23},\quad\xij{13}\xij{12}+\xij{12}\xij{23}+\xij{23}\xij{13}
  \end{align*}
  for all $i,j$. $\BV(V)$ and $\BV(\oV)$ are both isomorphic to $\FK_3$. 
  \smallskip
  \item\label{item:triangular decomposition} $\D$ is generated as an algebra by $\DSn$, $V$ and $\oV$. We have a triangular decomposition
  \begin{align*}
  \BV(V)\ot\DSn\ot\BV(\oV)\longrightarrow\D,
 \end{align*}
 {\it i.e.} the multiplication induces a linear isomorphism. Moreover, $\D$ is generated by $\DSn$, $\xij{12}$ and $\yij{12}$ because $\Sn_3$ acts transitively in the bases of $V$ and $\oV$.
  \item\label{property:graded D} $\D$ is a graded algebra with $\deg V=-1$, $\deg\DSn=0$ and $\deg\oV=1$.
   \smallskip
  \item\label{property:graded subalgebras} $\D^{\leq0}=\BV(V)\#\DSn$, $\D^{\geq0}=\BV(\oV)\#\DSn$ and $\DSn$ are graded Hopf subalgebras \cite[Lemma 8 and (25)]{PV2} and \cite[Lemma 4.3]{vay-mca17}.
   \smallskip
  \item\label{item:comultiplication} The comultiplication of $\D$ is completely determined by
  \begin{align*}
  &\Delta(g)=g\ot g\quad\mbox{and}\quad\Delta(\delta_g)=\sum_{h\in G}\delta_h\ot\delta_{h^{-1}g}\quad\mbox{for all $g\in\Sn_3$};\\
  \Delta(\xij{ij})&=\xij{ij}\ot1+(ij)\ot\xij{ij}\quad\mbox{and}\\
  \Delta(\yij{ij})&=\yij{ij}\ot1+\sum_{g\in\Sn_3}\sgn(g)\delta_g\ot y_{g^{-1}(ij)g}\quad\mbox{for all transpositions $(ij)\in\Sn_3$}.
  \end{align*}
\end{enumerate}

\begin{obs}
$\D$ is a spherical Hopf algebra \cite{spherical1}. The pivot is the $\sgn$ representation. Explicitly, $\sgn=\sum_{g\in\Sn_3}\sgn(g)\,\delta_g\in\DSn$. 

In fact, it is an involution and it is easy to check that $\cS^2(h)=\sgn\cdot h\cdot\sgn$ $\forall h\in\D$.
\end{obs}

\subsection{Graded Modules}

Our objects of study are the finite-dimensional $\Z$-graded left modules over $\D$, {\it graded modules} for short. We will consider the graded $\D$-modules as graded modules over $\D^{\leq0}$, $\D^{\geq0}$ and $\DSn$  by restricting the action. In particular, the following is a direct consequence of \ref{property:graded subalgebras} above and it is a particular case of \cite[Proposition 5.2]{vay-mca17}.

\begin{lema}\label{le:forgetful functor}
The restriction of scalars is a monoidal functor from the category of graded $\D$-modules to the category of graded $\DSn$-modules.\qed 
\end{lema}

By the definition of Nichols algebra, $\BV(V)$ (resp. $\BV(\oV)$) is an algebra in the category of graded $\DSn$-modules and, as we mention in \ref{property:graded subalgebras}, $\D^{\leq0}$ (resp. $\D^{\geq0}$) is the corresponding bosonization. Then the following is clear, see for instance \cite[(31)]{PV2}.

\begin{lema}
The category of graded $\D^{\leq0}$-modules (resp. $\D^{\geq0}$-modules) is equivalent to the category of graded $\BV(V)$-modules (resp. $\BV(\oV)$-modules) in the category of graded $\DSn$-modules. \qed
\end{lema}

We will use the next consequence of this lemma.

\begin{obs}\label{obs:action maps}
Let $\fN=\oplus_{i\in\Z}\fN(i)$ be a graded module. Then the action maps $V\ot\fN(i)\longrightarrow\fN(i-1)$ and $\oV\ot\fN(i)\longrightarrow\fN(i+1)$ are morphisms of $\DSn$-modules. 
\end{obs}

\subsection{Weights}\label{subsec:weights}
As we mention in the introduction the simple $\DSn$-modules are the {\it weights} in our context. These are parametrized by
\begin{align*}
\Lambda=\left\{\varepsilon=(e,+),\, (e,-),\, (e,\rho),\, (\sigma,+),\, (\sigma,-),\, (\tau,0),\, (\tau,1),\, (\tau,2)\right\},
\end{align*}
recall Table \ref{tab:weight}. We give their explicit structure in the appendix. We identify every {\it weight} $\lambda\in\Lambda$ with the simple $\DSn$-module $M(\lambda)$ in \cite[\S 5.2]{PV2}. 

A {\it highest} (resp. {\it lowest}) {\it weight} is a weight $\lambda$ which also is a simple module over $\D^{\geq0}$ with $\oV\cdot\lambda=0$ (resp. $\D^{\leq0}$ with $V\cdot\lambda=0$). A {\it highest-weight} (resp. {\it lowest-weight}) {\it module} is a $\D$-module generated by a {\it highest-weight} (resp. {\it lowest-weight}).

\subsection{Characters} 

The Grothendieck ring of the category of $\DSn$-modules is the abelian group $K=\Z\Lambda$ endowed with the product $\lambda\cdot\mu=M(\lambda)\ot M(\mu)$ and unit $\varepsilon$. These tensor products were explicitly given in \cite[\S 5.2.4]{PV2}. We will often use these fusion rules in the coming section. The Grothendieck ring of $\D(G)$, for any finite group $G$, was described in \cite{MR1367852}. Given a $\DSn$-module $N$, the {\it character} $\ch N$ is the representative of $N$ in the Grothendieck ring $K$. 

We shall consider $\DSn$ as a graded algebra concentrated in degree zero. If $N$ is a graded $\DSn$-module, we denote $N(i)$ its homogeneous component of degree $i$. The {\it shift of grading functor} $[1]$ is defined by $N[1](i)=N(i-1)$. Thus, the Grothendieck ring $K^{\bullet}$ of the category of $\DSn$-modules is a $\Z[t,t^{-1}]$-algebra if we identify $t^{\pm1}$ with $\varepsilon[\pm1]$. Therefore $K^{\bullet}=K[t,t^{-1}]$ via the {\it graded character} 
$$
\chgr N=\sum_{i\in\Z}\ch N(i)\,t^i\in K[t,t^{-1}].
$$
For instance,
\begin{align}\label{eq:character of FK3}
\chgr\BV(V)=\chgr\BV(\oV)=\varepsilon+(\sigma,-)\,t^{-1}+\left((\tau,1)+(\tau,2)\right)\,t^{-2}+(\sigma,-)\,t^{-3}+\varepsilon\,t^{-4}.
\end{align}
Notice that there exist polynomials $p_{N,\lambda}\in\Z[t,t^{-1}]$ such that 
$$
\chgr N=\sum_{\lambda\in\Lambda}\,p_{N,\lambda}\,\lambda\Longleftrightarrow N\simeq\oplus_{\lambda\in\Lambda}\,p_{N,\lambda}\cdot\lambda.
$$
We have that $N^*(i)=(N(-i))^*$ and hence $\chgr N^*=\sum_{\lambda\in\Lambda}\,\overline{p_{N,\lambda}}\,\lambda^*$ where $\overline{p(t,t^{-1})}$ $=p(t^{-1},t)$ for any $p\in\Z[t,t^{-1}]$.

We will often use the following consequence of Lemma \ref{le:forgetful functor}. 

\begin{lema}\label{le:chgr es morf de anillos}
Let $R^{\bullet}$ be the Grothendieck ring of the category of graded modules. Then $\chgr:R^\bullet\longrightarrow K[t,t^{-1}]$ is a morphism of $\Z[t,t^{-1}]$-algebras. \qed
\end{lema}

\subsection{The simple modules}

Given $\lambda\in\Lambda$, $\fL(\lambda)$ denotes the unique simple module of highest-weight $\lambda$. This is graded and every simple module is isomorphic to some $\fL(\lambda)$ \cite[Theorem 3]{PV2}. The simple modules also are distinguished by their lowest-weights \cite[Theorem 4]{PV2}. In \cite{PV2} we have studied the simple $\D$-modules in details. Their graded characters are
\begin{align*}
 \chgr\fL(\varepsilon)&=\varepsilon,\\
\noalign{\smallskip}
 \chgr\fL(e,\rho)&=(e,\rho)+(\sigma,+)\,t^{-1}+(\tau,0)\,t^{-2},\\
 \noalign{\smallskip}
  \chgr\fL(\tau,0)&=(\tau,0)+(\sigma,+)\,t^{-1}+(e,\rho)\,t^{-2},\\
 \noalign{\smallskip}
 \chgr\fL(\sigma,-)&=(\sigma,-)+\left((\tau,1)+(\tau,2)\right)\,t^{-1}+(\sigma,-)\,t^{-2},\\
 \noalign{\smallskip}
 \chgr\fL(\lambda)&=\lambda\cdot\chgr\BV(V),\quad\forall\,\lambda\in\Lsp:=\{(e,-),(\sigma,+),(\tau,1),(\tau,2)\}.
\end{align*}

By \cite[Theorem 6]{PV2} and \cite[Corollary 17]{vay-proj}, $\fL(\lambda)$ is projective (and injective because any finite-dimensional Hopf algebra is Frobenius) if and only if $\lambda\in\Lsp$.
The remainder simple modules generate a single block of the category of $\D$-modules because they are composition factors of an indecomposable module, the Verma module of $(\sigma,-)$ \cite[Theorem 7]{PV2}. Verma modules are recalled in the next subsection.

\begin{obs}\label{obs:ungraded character}
As (ungraded) $\DSn$-modules $\fL(e,\rho)\simeq\fL(\tau,0)$, that is they have identical (ungraded) character but they are not isomorphic as $\D$-modules.
\end{obs}

These simple modules are self-dual except for
\begin{align*}
\fL(e,\rho)^*\simeq\fL(\tau,0)\quad\mbox{and}\quad\fL(\tau,1)^*\simeq\fL(\tau,2).
\end{align*}
Let $\overline{\lambda}$ denote the lowest-weight of $\fL(\lambda)$. Then
\begin{align}\label{eq:lambda barra}
\overline{(e,\rho)}=(\tau,0),\quad\overline{(\tau,0)}=(e,\rho)\quad\mbox{and}\quad \overline{\lambda}=\lambda
\quad\forall\,\lambda\in\Lambda\setminus\{(e,\rho),(\tau,0)\}.
\end{align}

\begin{obs}
The anonymous referee informs us that the bijection in \eqref{eq:lambda barra} corresponds to the unique non-trivial braided autoequivalence of the category of $\DSn$-modules by \cite[\S 6.6]{LENTNER2017264} and \cite[\S 8.1]{NIKSHYCH2014191}. It will be interesting to know whether this is a general fact. More precisely, let $\D$ be the Drinfeld double of the bosonization $\BV(V)\#H$ of a finite-dimensional Nichols algebra and a finite-dimensional semisimple Hopf algebra. Does ``picking the lowest-weight of a simple highest-weight module'' define a braided autoequivalences of the category of $\D(H)$-modules? We hope to address this question in future works.
\end{obs}

\

The set $\left\{\chgr\fL(\lambda)\mid\lambda\in\Lambda\right\}$ is a $\Z[t,t^{-1}]$-basis of $R^{\bullet}$ by \cite[Theorem 9]{vay-proj}. Then, for every graded module $\fN$ there are unique Laurent polynomials $p_{\fN,\fL(\lambda)}$ such that
$$\chgr\fN=\sum_{\lambda\in\Lambda}p_{\fN,\fL(\lambda)}\chgr\fL(\lambda).$$
We can deduce the following information from these polynomials.

\begin{obs}\label{obs:sobre los subcocientes}
Assume that $p_{\fN,\fL(\lambda)}=\sum a_{\fN,\fL(\lambda),i}\,t^i$ with $a_{\fN,\fL(\lambda),i}\neq0$.
\begin{enumerate}[label=(\roman*)]
\smallskip
 \item $\fN$ has $a_{\fN,\fL(\lambda),i}$ composition factors isomorphic to $\fL(\lambda)[i]$.
 \smallskip
 \item  If $\fL(\lambda)$ is projective, then $a_{\fN,\fL(\lambda),i}\,\fL(\lambda)[i]$ is a direct summand of $\fN$.
 \smallskip
 \item There exists a weight $S\subset\fN(i)$ isomorphic to $\lambda$ such that $\D S/\fX\simeq\fL(\lambda)[i]$ for some maximal submodule $\fX$ of $\D S$.
\smallskip 
 \item Let $S\subset\fN(i)$ be a weight isomorphic to $\lambda$ and $\fX$ a maximal submodule of $\D S$. Then $\D S/\fX\simeq\fL(\mu)[j]$ such that $a_{\fN,\fL(\mu),j}\neq0$ and $\lambda[i]$ is a weight of $\fL(\mu)[j]$. In particular, $\D S/\fX\simeq\fL(\lambda)[i]$ if $\lambda[i]$ is not a weight of any composition factor $\fL(\mu)[j]$ of $\fN$ with $\mu\neq\lambda$ or $j\neq i$.
\end{enumerate}

In fact, (i), (ii) and (iii) are clear, cf. \cite[\S 3.2]{vay-proj}. Since every composition factor of $\D S$ is a composition factor of $\fN$, (iv) holds.
\end{obs}

\subsection{Verma modules} 

Given a highest-weight $\lambda\in\Lambda$, the induced module
\begin{align}\label{eq:verma}
\fM(\lambda)=\D\ot_{\D^{\geq0}}\lambda\simeq\BV(V)\ot\lambda                                                                                                                                                   
\end{align}
is called {\it Verma module}; where the isomorphism is of $\Z$-graded $\D^{\leq0}$-modules. This is the universal highest-weight module of weight $\lambda$. Its head is isomorphic to $\fL(\lambda)$ and its socle is $\fL(\mu)$ with $\overline{\mu}=\BV^{n_{top}}(V)\ot\lambda$ \cite[Theorems 3 and 4]{PV2}.

In this case, the Verma modules are self-dual except $\fM(\tau,1)^*\simeq\fM(\tau,2)$ by \cite[(10)]{vay-proj} since $\lambda_V=\ch\BV^{n_{top}}(V)=\varepsilon$. By \cite[Theorem 6]{PV2}, $\fM(e,-)$, $\fM(\sigma,+)$, $\fM(\tau,1)$ and $\fM(\tau,2)$ are simple and hence they are projective by \cite[Corollary 17]{vay-proj}. Their graded characters are
\begin{align*}
 \chgr\fM(\varepsilon)&=(1+t^{-4})\chgr\fL(\varepsilon)+t^{-1}\chgr\fL(\sigma,-),\\
 \noalign{\smallskip}
 \chgr\fM(e,\rho)&=\chgr\fL(e,\rho)+t^{-1}\chgr\fL(\sigma,-)+t^{-2}\chgr\fL(\tau,0),\\
 \noalign{\smallskip}
  \chgr\fM(\tau,0)&=\chgr\fL(\tau,0)+t^{-1}\chgr\fL(\sigma,-)+t^{-2}\chgr\fL(e,\rho),\\
   \noalign{\smallskip}
   \chgr\fM(\sigma,-)&=(1+t^{-2})\chgr\fL(\sigma,-)+\\
 &\qquad\qquad+t^{-1}\chgr\fL(e,\rho)+t^{-1}\chgr\fL(\tau,0)+(t^{-1}+t^{-3})\chgr\fL(\varepsilon),\\
 \noalign{\smallskip}
 \chgr\fM(\lambda)&=\chgr\fL(\lambda),\quad\forall\,\lambda\in\Lsp.
\end{align*}
In fact, we can calculate explicitly the polynomials $p_{\fM(\mu),\fL(\lambda)}$ or use \cite[Theorems 7, 8, 9 and 10]{PV2} where 
we have computed the lattice of submodules of the Verma modules.

\subsection{co-Verma modules} 

Given a lowest-weight $\mu$, the induced module
\begin{align}\label{eq:co verma}
\fW(\mu)=\D\ot_{\D^{\leq0}}\mu\simeq\BV(\oV)\ot\mu.                                                                                                                                                    
\end{align}
is called {\it co-Verma module}; the isomorphism is of $\Z$-graded $\D^{\geq0}$-modules. By \cite[Theorem 10]{vay-proj} we know that $\chgr\fW(\lambda)=t^4\chgr\fM(\lambda)$. Hence
\begin{align*}
 \chgr\fW(\varepsilon)&=(1+t^{4})\chgr\fL(\varepsilon)+t^3\chgr\fL(\sigma,-),\\
 \noalign{\smallskip}
 \chgr\fW(e,\rho)&=t^2\chgr\fL(\tau,0)+t^3\chgr\fL(\sigma,-)+t^{4}\chgr\fL(e,\rho),\\
 \noalign{\smallskip}
  \chgr\fW(\tau,0)&=t^2\chgr\fL(e,\rho)+t^{3}\chgr\fL(\sigma,-)+t^{4}\chgr\fL(\tau,0),\\
 \noalign{\smallskip}
 \chgr\fW(\sigma,-)&=(t^2+t^{4})\chgr\fL(\sigma,-)+\\
 &\qquad\qquad+t^{3}\chgr\fL(\tau,0)+t^{3}\chgr\fL(e,\rho)+(t+t^{3})\chgr\fL(\varepsilon),\\
 \noalign{\smallskip}
 \chgr\fW(\lambda)&=t^4\chgr\fL(\lambda),\quad\forall\,\lambda\in\Lsp.
\end{align*}

As the Verma modules, the co-Verma modules have simple head and simple socle.

\begin{lema}\phantomsection\label{le:coVerma}
\begin{enumerate}[label=(\roman*)]
 \item The socle of $\fW(\lambda)$ is isomorphic to $\fL(\lambda)$ for all $\lambda\notin\Lsp$.
 \smallskip
 \item The head of $\fW(\lambda)$ is isomorphic to $\fL(\overline{\lambda})$ for all $\lambda\notin\Lsp$.
 \smallskip
 \item The socle of $\fW(\lambda)/\soco\fW(\lambda)$ is isomorphic to $\fL(\sigma,-)$ if $(\sigma,-)\neq\lambda\notin\Lsp$.
 \smallskip
 \item The socle of $\fW(\sigma,-)/\soco\fW(\sigma,-)$ is isomorphic to $2\fL(\varepsilon)\oplus\fL(e,\rho)\oplus\fL(\tau,0)$.
  \item The unique maximal submodule of $\fW(\lambda)$ is the preimage of the socle of $\fW(\lambda)/\soco\fW(\lambda)$ for all $\lambda\notin\Lsp$.
 \smallskip
 \item 
 $\fW(\lambda)\simeq\fM(\lambda)$ for all $\lambda\in\Lsp$. In particular, they are simple and projective.
\end{enumerate}
\end{lema}

\begin{proof}
(i), (ii) and (vi) follow from \cite[(15)]{vay-proj}.

(iii) If $(\sigma,-)\neq\lambda\notin\Lsp$, then $\fW(\lambda)$ has three composition factors because of the graded character. Then, by (i) and (ii), the socle of $\fW(\lambda)/\soco\fW(\lambda)$ is simple and isomorphic to $\fL(\sigma,-)$. 

Notice that (v) follows from (iii) for $\lambda\neq(\sigma,-)$.

(iv) Let $\lambda\in\{\varepsilon,(e,\rho),(\tau,0)\}$. Then $\lambda$ is contained in the maximal submodule of $\fW(\sigma,-)$ because of $\chgr\fL(\sigma,-)$, which is the head of $\fW(\sigma,-)$ by (ii). In particular, if the degree of $\lambda$ is $1$, then $\lambda$ is a lowest-weight. Hence the submodule $\D\lambda$ of $\fW(\sigma,-)$ is a quotient of the co-Verma module $\fW(\lambda)$. Therefore the socle of $\fW(\sigma,-)/\soco\fW(\sigma,-)$ contains a submodule isomorphic to $\fL(\varepsilon)\oplus\fL(e,\rho)\oplus\fL(\tau,0)$ by (i)--(iii). 

The other copy of $\fL(\varepsilon)$ corresponds to the weight $\varepsilon$ of degree $3$, see $\chgr\fW(\sigma,-)$. In fact, it is a highest-weight in $\fW(\sigma,-)/\soco\fW(\sigma,-)$, since this quotient has no homogeneous component of degree $4$. Then the submodule $\D\varepsilon$ of $\fW(\sigma,-)/\soco\fW(\sigma,-)$ is a quotient of the Verma module $\fM(\varepsilon)$. We claim that $\D\varepsilon\simeq\fL(\varepsilon)$ and (iv) follows. Otherwise, $\D\varepsilon$ has a composition factor isomorphic to $\fL(\sigma,-)$ by \cite[Theorem 8]{PV2}. From $\chgr\fW(\sigma,-)$ we deduce that this composition factor corresponds to the head of $\fW(\sigma,-)$. But this is not possible because $\varepsilon$ is contained in the maximal submodule of $\fW(\sigma,-)$. This finishes the proof of (iv) which implies (v).
\end{proof}

\section{Extensions of simple modules}\label{sec:ext simple modules}

In this section, we classify the extensions of $\fL(\lambda)$ by $\fL(\mu)$ for $\lambda,\mu\in\Lambda\setminus\Lambda_{sp}$, {\it i.e.} the modules $\fE$ which fits into a short exact sequence of the form
\begin{align}\label{eq:extension}
0\longrightarrow \fL(\mu)\overset{i}{\longrightarrow}\fE\overset{\pi}{\longrightarrow} \fL(\lambda)\longrightarrow0.
\end{align}
We say that the extension is trivial if $\fE\simeq\fL(\mu)\oplus\fL(\lambda)$. If $\lambda\in\Lambda_{sp}$ or $\mu\in\Lambda_{sp}$, then $\fE$ is trivial because $\fL(\lambda)$ is injective (resp. $\fL(\mu)$ is projective).

\begin{lema}\label{le:trivial extensions}
If $\lambda,\mu\in\{\varepsilon,(e,\rho),(\tau,0)\}$ or $\lambda=\mu=(\sigma,-)$, then $\fE\simeq\fL(\mu)\oplus\fL(\lambda)$.
\end{lema}

\begin{proof}
Since $\D$ is finite-dimensional, the space of extensions and the space of graded extensions are isomorphic, see for instance \cite[Corollary 2.4.7]{MR2046303}. Thus, we can assume that $\fE$ fits into a short exact sequence of the form
\begin{align*}
0\longrightarrow \fL(\mu)[\ell]\overset{i}{\longrightarrow}\fE\overset{\pi}{\longrightarrow} \fL(\lambda)\longrightarrow0.
\end{align*}
In particular, $\fE\simeq\fL(\mu)[\ell]\oplus\fL(\lambda)$ as graded $\DSn$-modules. Let $\iota$ be a section of $\pi$ as graded $\DSn$-modules.

We first show case-by-case that either $\iota(\lambda)$ is a highest-weight of $\fE$ or $\iota(\overline{\lambda})$ is a lowest-weight of $\fE$; recall that $\overline{\lambda}$ denotes the lowest-weight of $\fL(\lambda)$. Recall that the restrictions of the action maps $\oV\ot\iota(\lambda)\longrightarrow\fL(\mu)(1-\ell)$ and $V\ot\iota(\overline{\lambda})\longrightarrow\fL(\mu)(-1-\ell)$ are morphisms of graded $\DSn$-modules (Remark \ref{obs:action maps}) and keep in mind the character of $\fL(\mu)$ and the fusion rules for the simple $\DSn$-modules given in \cite[\S 2.5.4]{PV2}.

If $\lambda=\varepsilon$, then $\oV\cdot\iota(\lambda)=0=V\cdot\iota(\overline{\lambda})$ because $(\sigma,-)$ is not a weight of $\fL(\mu)$.

If $\lambda=(e,\rho)$ and $\oV\cdot\iota(\lambda)\neq0$, then $\mu\neq\varepsilon$ and $\oV\cdot\iota(\lambda)\simeq(\sigma,+)\simeq\fL(\mu)(-1)$. This forces $\ell=2$. Hence $V\cdot\iota(\overline{\lambda})=0$ because $\fL(\mu)[-2](-3)=\fL(\mu)(-5)=0$. The case $\lambda=(\tau,0)$ is analogous.

If $\lambda=(\sigma,-)$ and $\oV\cdot\iota(\lambda)\neq0$, then $\oV\cdot\iota(\lambda)\subseteq\fL(\sigma,-)(-1)$ and we can conclude that $V\cdot\iota(\overline{\lambda})=0$ as above.

Now, we have that the submodule $\fN$ generated by $\iota(\lambda)$ is a graded quotient of either $\fM(\lambda)$ or $\fW(\overline{\lambda})$. We know the graded quotients of $\fM(\lambda)$ and $\fW(\overline{\lambda})$ from \cite[\S 4]{PV2} and Lemma \ref{le:coVerma}. By the graded characters of these quotients, we deduce that $\fN\simeq\fL(\lambda)$ and hence the lemma follows.
\end{proof}

For $\lambda\in\{\e,(e,\rho),(\tau,0)\}$, we have distinguished extensions thanks to \cite[Theorems 9 and 10]{PV2} and Lemma \ref{le:coVerma}. Namely,
\begin{align*}
\quad&0\longrightarrow \fL(\sigma,-)[-1]\longrightarrow\fM(\lambda)/\soco\fM(\lambda)\longrightarrow\fL(\lambda)\longrightarrow0\quad\mbox{and}\\
\noalign{\smallskip}
\quad&0\longrightarrow \fL(\sigma,-)[3]\longrightarrow\fW(\overline{\lambda})/\soco\fW(\overline{\lambda})\longrightarrow\fL(\lambda)[2]\longrightarrow0.
\end{align*}

\begin{definition}\label{def:Est}
Let $s,t$ be scalars and $\lambda\in\{(e,\rho),(\tau,0)\}$. We set $\fE(\lambda)_{0,0}=\fL(\sigma,-)\oplus\fL(\lambda)$ for $s=t=0$. For non-zero scalars, the Baer sum of the above extensions will be denoted by
$$
\fE(\lambda)_{s,t}=s\,\bigl(\fM(\lambda)/\soco\fM(\lambda)\bigr)+t\,\bigl(\fW(\overline{\lambda})/\soco\fW(\overline{\lambda})\bigr).
$$
\end{definition}

Therefore $\fE(\lambda)_{s,t}$ is an extension of $\fL(\lambda)$ by $\fL(\sigma,-)$ and its dual $\fE(\lambda)_{s,t}^*$ is an extension of $\fL(\sigma,-)$ by $\fL(\overline{\lambda})$.

\begin{lema}\label{le:notrivial extension}
Let $\fE$ be an extension of $\fL(\lambda)$ by $\fL(\mu)$.
\begin{enumerate}[label=(\roman*)]
 \item If $\lambda\in\{(e,\rho),(\tau,0)\}$ and $\mu=(\sigma,-)$, then $\fE\simeq\fE(\lambda)_{s,t}$ for some $s,t\in\ku$.
Moreover, it is a graded extension if and only if it is isomorphic (up to shifts) to either $\fE(\lambda)_{1,0}$, $\fE(\lambda)_{0,1}$ or $\fE(\lambda)_{0,0}$.
 \smallskip
 
 \item If $\lambda=(\sigma,-)$ and $\mu\in\{(e,\rho),(\tau,0)\}$, then $\fE\simeq\fE(\lambda)_{s,t}^*$ for some $s,t\in\ku$.
Moreover, it is a graded extension if and only if it is isomorphic (up to shifts) to either $\fE(\lambda)_{1,0}^*$, $\fE(\lambda)_{0,1}^*$ or $\fE(\lambda)_{0,0}^*$.
\end{enumerate}
\end{lema}

\begin{proof}
(i) We prove only the case $\lambda=(e,\rho)$. The case $\lambda=(\tau,0)$ is similar. As in the above lemma, it is enough to prove that if $\fE$ is a nontrivial graded extension, then $\fE\simeq\fE(\lambda)_{1,0}$ or $\fE\simeq\fE(\lambda)_{0,1}$. Thus, we can assume that $\fE$ fits into a short exact sequence of the form
\begin{align*}
0\longrightarrow \fL(\sigma,-)[\ell]\overset{i}{\longrightarrow}\fE\overset{\pi}{\longrightarrow} \fL(e,\rho)[2]\longrightarrow0.
\end{align*}
In particular, $\fE\simeq\fL(\sigma,-)[\ell]\oplus\fL(e,\rho)[2]$ as graded $\DSn$-modules. Let $\iota$ be a section of $\pi$ as graded $\DSn$-modules. The action $\oV\ot\iota(e,\rho)\longrightarrow\fL(\sigma,-)(3-\ell)$ is a morphism of graded $\DSn$-modules. If $\iota(e,\rho)$ is a highest-weight in $\fE$, then $\fE\simeq\fE(e,\rho)_{1,0}$. Otherwise, $\oV\cdot\iota(e,\rho)\simeq(\sigma,-)$ is homogeneous. Then $\ell=3$ or $\ell=5$ by $\chgr\fL(\sigma,-)$. On the other hand, $V\cdot\iota(\tau,0)\subseteq\fL(\sigma,-)(-1-\ell)$. Thus, $\iota(\tau,0)$ is a lowest-weight of $\fE$ and hence $\fE\simeq\fE(e,\rho)_{0,1}$. Moreover, this forces that $\ell=3$ because $\chgr\fE(e,\rho)_{0,1}=t^2\chgr\fL(e,\rho)+t^3\chgr\fL(\sigma,-)$.

(ii) is equivalent to (i) because $(\sigma,-)^*=(\sigma,-)$ and $\mu^*\in\{(e,\rho),(\tau,0)\}$.
\end{proof}

In \cite[Lemma 26]{PV2} we found a family of submodules of $\fM(\sigma,-)$ which are extensions of $\fL(\e)$ by $\fL(\sigma,-)$. Among these, $\fT_{0,1}$ is graded but is not neither a highest-weight module nor a lowest-weight module. We next give the actions of $\DSn$, $\xij{12}$ and $\yij{12}$ over it, see the proof of \cite[Lemma 26]{PV2}. By \S\ref{subsec:lo basico} \ref{item:triangular decomposition}, its graded $\D$-structure is completely determined by this datum.

\begin{definition}
As graded $\DSn$-module, $\fT_{0,1}$ is isomorphic to $\fL(\sigma,-)[1]\oplus \e$. We let $\fL(\sigma,-)[1]$ be a graded $\D$-submodule of $\fT_{0,1}$ with basis $\{c_i\}_{i=1}^{10}$ as in the appendix. The generator of the graded $\DSn$-submodule $\e$ of $\fT$ is denoted by $\ft_{0,1}$. The elements $\xij{12}$ and $\yij{12}$ act over $\ft_{0,1}$ as follows
\begin{align}\label{eq:action on ft01}
x_{(12)}\cdot \ft_{0,1}=c_{1}\quad\mbox{and}\quad y_{(12)}\cdot \ft_{0,1}=c_{8}. 
\end{align}
\end{definition}

Hence $V\cdot\e$ is the lowest-weight of $\fL(\sigma,-)[1]$ and $\oV\cdot\e$ is the highest-weight of $\fL(\sigma,-)[1]$ (cf. Appendix) and therefore $\fT_{0,1}/\fL(\sigma,-)[1]\simeq\fL(\e)$ as graded modules.

\begin{definition}\label{def:ext de varepsilon}
Let $s,t,u$ be scalars. We set $\fE(\e)_{0,0,0}=\fL(\sigma,-)\oplus\fL(\e)$ for $s=t=u=0$. For non-zero scalars, we denote by $\fE(\e)_{s,t,u}$ the Baer sum of extensions
\begin{align*}
\fE(\e)_{s,t,u}=s\,\bigl(\fM(\e)/\soco\fM(\e)\bigr)+t\,\bigl(\fW(\e)/\soco\fW(\e)\bigr)\,+u\,\fT_{0,1}.
\end{align*}
\end{definition}

Therefore $\fE(\e)_{s,t,u}$ is an extension of $\fL(\e)$ by $\fL(\sigma,-)$ and its dual $\fE(\e)_{s,t,u}^*$ is an extension of $\fL(\sigma,-)$ by $\fL(\e)$.

\begin{lema}\label{le:notrivial extension varepsilon}
\begin{enumerate}[label=(\roman*)]
 \item Let $\fE$ be an extension of $\fL(\e)$ by $\fL(\sigma,-)$, then $\fE\simeq\fE(\e)_{s,t,u}$ for some $s,t,u\in\ku$.
Moreover, it is a graded extension if and only if it is isomorphic (up to shifts) to either $\fE(\e)_{1,0,0}$, $\fE(\e)_{0,1,0}$, $\fE(\e)_{0,0,1}$ or $\fE(\e)_{0,0,0}$.
 \smallskip
 
 \item Let $\fE$ be an extension of $\fL(\sigma,-)$ by $\fL(\e)$, then $\fE\simeq\fE(\e)_{s,t,u}^*$ for some $s,t,u\in\ku$.
Moreover, it is a graded extension if and only if it is isomorphic (up to shifts) to either $\fE(\e)_{1,0,0}^*$, $\fE(\e)_{0,1,0}^*$, $\fE(\e)_{0,0,1}^*$ or $\fE(\e)_{0,0,0}$.
\end{enumerate}
 \end{lema}

\begin{proof}
(i) As in the previous lemma, it is enough to consider the graded case. We can assume that $\fE\simeq\fL(\sigma,-)[\ell]\oplus\e$ as graded $\DSn$-modules and $\fL(\sigma,-)[\ell]$ is a graded $\D$-submodule of $\fE$. If $\e$ is either a highest-weight or a lowest-weight, then $\fE$ is isomorphic to either $\fE(\lambda)_{1,0,0}$, $\fE(\lambda)_{0,1,0}$ or $\fE(\lambda)_{0,0,0}$. 

Suppose now $V\cdot\e\neq0$ and $\oV\cdot\e\neq0$. This forces that: $\ell=1$, $V\cdot\e$ is the lowest-weight of $\fL(\sigma,-)[1]$ and $\oV\cdot\e$ is the highest-weight of $\fL(\sigma,-)[1]$. Let $\ft$ be a generator of the weight $\e$ of $\fE$. To complete the proof, we shall check that \eqref{eq:action on ft01} holds up to a change of basis. 

We can use the basis of $\fL(\sigma,-)$ given in the appendix. Thus, the lowest-weight $\nu$ of $\fL(\sigma,-)[1]$ is spanned by $\{c_1,c_2,c_3\}$ and $\delta_{(12)}\cdot\nu=\ku c_1$. Since the action is a morphism of $\DSn$-modules, $\delta_{(12)}\cdot(\xij{12}\cdot\ft)=\xij{12}\cdot\ft$. Then, $\xij{12}\cdot\ft=rc_1$ for some $0\neq r\in\ku$. Notice that $\xij{12}\cdot\ft=0$ implies $V\cdot\e=0$ because $\Sn_3$ acts transitively on the $\xij{ij}$'s. Similarly we deduce that $\yij{12}\cdot\ft=vc_8$ for some $0\neq v\in\ku$. 

On the other hand, from the defining relations of $\D$ (cf. \cite[page 427]{PV2}), we deduce that
\begin{align*}
(23)\yij{23}\yij{13}\xij{12}=&(23)\xij{12}\yij{13}\yij{23}\\
-(23)(12)&(\delta_{(23)}-\delta_{(23)(12)})\yij{23}-(23)\yij{23}(12)(\delta_{(13)}-\delta_{(13)(12)}).
\end{align*}
We compute the action of both side on $\ft$. First, using the appendix, we have that
$$
(23)\yij{23}\yij{13}\xij{12}\cdot\ft=-rc_8.
$$
Next, the first term of the right hand acts by zero because $\yij{23}\cdot\ft$ is in the highest-weight of $\fL(\sigma,-)$. Also, the last term acts by zero because $\delta_g\cdot\ft=\delta_{g,e}\ft$. Finally,
$$
-(23)(12)(\delta_{(23)}-\delta_{(23)(12)})\yij{23}\cdot\ft=-\yij{12}\cdot\ft=-vc_8.
$$
Hence $r=v\neq0$. Therefore, if we change $\ft$ by $\frac{1}{r}\ft$, we have that $\fE\simeq\fT_{0,1}$ as desired.

(ii) follows from (i) by duality.
\end{proof}

By the above lemmas the separated quiver of $\D$ is given by Figure \ref{fig:separated quiver}. Then, we deduce the following proposition, see for instance \cite[\S 4.2]{dim72} for details.

\begin{prop}\label{prop:wild}
$\D$ is of wild representation type.\qed
\end{prop}

\section{The tensor products of non-projective simple modules}\label{sec:tensores}

In this section we describe the tensor products between the simple modules $\fL(\varepsilon)$, $\fL(e,\rho)$, $\fL(\tau,0)$ and $\fL(\sigma,-)$.

We will use the bases of the simple modules and the action over them given in the appendix. The action on the tensor product is induced by the comutilplication given in \S\ref{subsec:lo basico} \ref{item:comultiplication}. We will often use the fusion rules of the simple $\DSn$-modules given in \cite[\S 2.5.4]{PV2}.

\subsection{How to compute the indecomposable submodules}\label{subsec:key tools}

We explain the general strategy which we shall follow to compute the indecomposable summands. These ideas apply to any graded module $\fN=\oplus_{i\in\Z}\fN(i)$ over the Drinfeld double of a finite-dimensional Nichols algebra. See also \cite[\S 3.2]{PV2}

Assume that $\chgr\fN=\sum_{\lambda\in\Lambda}p_{\fN,\fL(\lambda)}\chgr\fL(\lambda)$ and $a_{\fN,\fL(\lambda),i}\neq0$. In view of Remark \ref{obs:sobre los subcocientes}, we shall start by computing the submodules $\D\lambda$ generated by the weights $\lambda\subset\fN(i)$. Among these, we will first consider the weights $\lambda$ such that $i$ is either maximal or minimal because this implies that $\D\lambda$ is a quotient of either the Verma module $\fM(\lambda)$ or the co-Verma module $\fW(\lambda)$. In fact, $\lambda$ will be either a highest or lowest weight. We know these quotients from \cite[\S 4]{PV2} and Lemma \ref{le:coVerma}, respectively.

For the remainder weights, we will repeatedly use that the action maps $V\ot\lambda\longrightarrow\fN(i-1)$ and $\oV\ot\lambda\longrightarrow\fN(i+1)$ are morphisms of $\DSn$-modules; this is Remark \ref{obs:action maps} with $\lambda$ instead of $\fN(i)$. Therefore $\D\lambda$ will be generated by the successive images of the former maps. We shall decompose $V\ot\lambda$ (respectively $\oV\ot\lambda$) into a direct sum of weights and apply the action on each summand. This restriction morphism will be zero or an injection by Schur Lemma. Hence it is enough to compute the action in a single element of each weight. The knowledge of $\ch\fN(i-1)$ (respectively $\ch\fN(i+1)$) will help to make less computations. 

Finally, we shall analyze the intersections of the submodules $\D\lambda$.

\subsection{The tensor product \texorpdfstring{$\fL(\tau, 0)\otimes \fL(e, \rho)$}{s}}

\begin{prop}\label{prop:1}
It holds that
\begin{align}\label{eq:fL tau 0 x fL e rho}
\chgr\left(\fL(\tau, 0)\otimes \fL(e, \rho)\right)=&\chgr\fL(\tau,1)+\chgr\fL(\tau,2)+t^{-2}\chgr\fL(\varepsilon).
\end{align}
Therefore $\fL(\tau, 0)\otimes \fL(e, \rho)\simeq\fL(\tau,1)\oplus\fL(\tau,2)\oplus\fL(\varepsilon)[-2]$ as graded modules.
\end{prop}

\begin{proof}
As $\chgr$ is a ring homomorphism and using the formulae of \cite[\S 5.2]{PV2}, we have that
\begin{align*}
\chgr\left(\fL(\tau, 0)\otimes \fL(e, \rho)\right)=&(\tau,0)(e,\rho)+t^{-1}(\sigma,+)\bigl((\tau,0)+(e,\rho)\bigr)+\\
&+t^{-2}\bigl((\tau,0)(\tau,0)+(\sigma,+)(\sigma,+)+(e,\rho)(e,\rho)\bigr)\\
&+t^{-3}(\sigma,+)\bigl((\tau,0)+(e,\rho)\bigr)+t^{-4}(\tau,0)(e,\rho)\\
=(\tau,1)+(\tau,2)+t^{-1}(\sigma,-)&\bigl((\tau,1)+(\tau,2)\bigr)+t^{-2}\bigl((\tau,1)+(\tau,2)\bigr)\bigl((\tau,1)+(\tau,2)\bigr)+\\
+t^{-2}\varepsilon+ t^{-3}&(\sigma,-)\bigl((\tau,1)+(\tau,2)\bigr)+t^{-4}\bigl((\tau,1)+(\tau,2)\bigr).
\end{align*}
Then, \eqref{eq:fL tau 0 x fL e rho} is a straightforward computation.

By \eqref{eq:fL tau 0 x fL e rho} and Remark \ref{obs:sobre los subcocientes}, the simple modules $\fL(\tau,1)$ and $\fL(\tau,2)$ are direct summands of $\fL(\tau, 0)\otimes \fL(e, \rho)$. Thus, the isomorphism holds because $\fL(\tau, 0)\otimes \fL(e, \rho)$ has only three composition factors.
\end{proof}

\begin{obs}
The weights $(\tau,1)$ and $(\tau,2)$ of the degree zero component are obviously highest-weights generating the simple submodules $\fL(\tau,1)$ and $\fL(\tau,2)$. The element generating the submodule $\fL(\varepsilon)$ is 
\begin{align*}
d=a_6\otimes b_2+a_7\otimes b_1 +a_3\otimes b_4+a_5\otimes b_3+a_4\otimes b_5+a_1\otimes b_6+a_2\otimes b_7,
\end{align*}
where the elements $a_i, b_j$ are presented in the appendix. In fact, using the appendix, we see that $y_{(12)}\cdot d=0$ and $x_{(12)}\cdot d=0$. 
\end{obs}

\subsection{The tensor product \texorpdfstring{$\fL(\sigma,-)\otimes \fL(\sigma,-)$}{s}}
As in \eqref{eq:fL tau 0 x fL e rho} we can see that
\begin{align}\label{eq:fL sigma menos x fL sigma menos}
\chgr\left(\fL(\sigma,-)\otimes\right.&\left. \fL(\sigma,-)\right)=\chgr\fL(\tau,1)+\chgr\fL(\tau,2)+2t^{-1}\chgr\fL(\sigma,-)+\\
\notag&+(1+2t^{-2}+t^{-4})\chgr\fL(\varepsilon)+(1+t^{-2})\bigl(\chgr\fL(e,\rho)+\chgr\fL(\tau,0)\bigr).
\end{align}
Therefore $\fL(\tau,1)$ and $\fL(\tau,2)$ are graded direct summands of $\fL(\sigma,-)\otimes \fL(\sigma,-)$ by Remark \ref{obs:sobre los subcocientes}. 
The aim  of this subsection is to show the next proposition. We give the proof after some preparatory lemmas. Recall the socle filtration $\{\soco^i\fA\}_{i\geq1}$ is given by the preimages of $\soco(\fA/\soco^{i-1}\fA)$ for $i>1$. 

\begin{prop}\label{prop: sigma menos por sigma menos}
There exists a graded indecomposable module $\fA$ with $\chgr\fA=$
\begin{align*}
=2t^{-1}\chgr\fL(\sigma,-)+(1+t^{-2}+t^{-4})\chgr\fL(\varepsilon)+(1+t^{-2})\bigl(\chgr\fL(e,\rho)+\chgr\fL(\tau,0)\bigr)
\end{align*}
such that $\fA^*\simeq\fA$ and 
$$\fL(\sigma,-)\otimes \fL(\sigma,-)\simeq\fL(\tau,1)\oplus\fL(\tau,2)\oplus\fL(\varepsilon)\oplus\fA.$$
Moreover
\begin{align*}
\operatorname{soc}\fA&=t^{-1}\fL(\sigma,-),\\
\soco^2\fA/\soco\fA&\simeq (1+t^{-2}+t^{-4})\fL(\varepsilon)\oplus(1+t^{-2})\fL(e,\rho)\oplus(1+t^{-2})\fL(\tau,0),\\
\soco^3\fA/\soco^2\fA&\simeq t^{-1}\fL(\sigma,-),\\
\soco^3\fA&=\fA.
\end{align*}
\end{prop}

Figure \ref{fig:fA} helps the reader to visualize the module $\fA$ and to follow the proof of the following lemmas.

\begin{figure}[h]
\begin{center}
\includegraphics{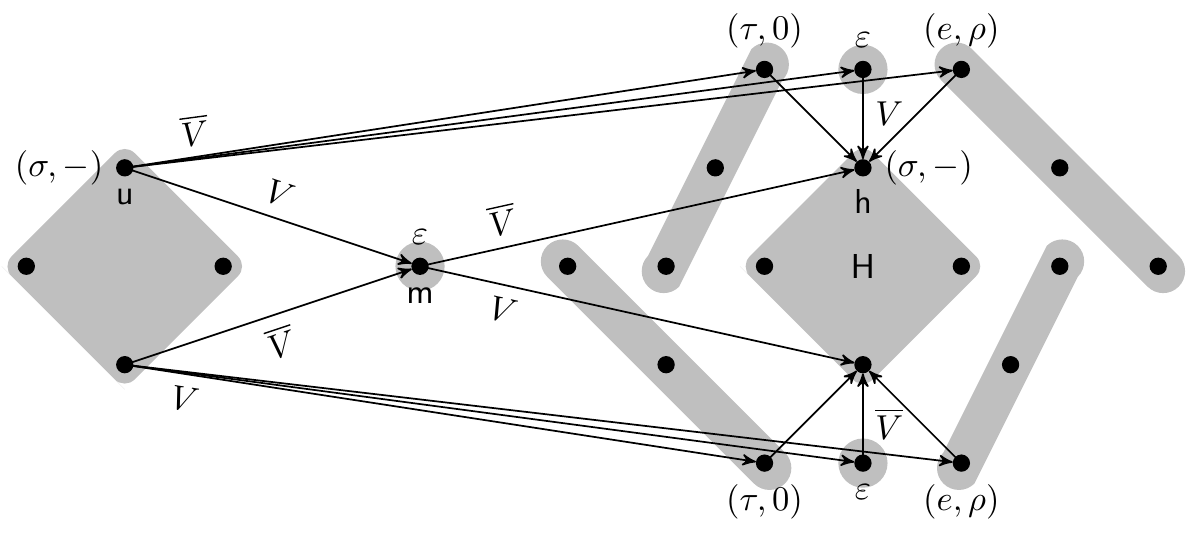}
\caption{The dots represent the weights of $\fA$. Each shadow region correspond to a composition factor whose highest-weight is in the top. The actions of $V$ and $\oV$ are illustrated by the arrows.}\label{fig:fA}
\end{center}
\end{figure}

By the fusion rules \cite[\S 2.5.4]{PV2}, $\fL(\sigma,-)\otimes \fL(\sigma,-)$ has four copies of the weight $\varepsilon$ in degree $-2$. In fact, these are
\begin{align*}
&\varepsilon_{-2,0}=c_1\ot c_8+c_2\ot c_9 + c_3\ot c_{10},&&\varepsilon_{-1,-1,1}=c_4\ot c_5+c_5\ot c_4,\\
&\varepsilon_{0,-2}=c_8\ot c_1+c_9\ot c_2 + c_{10}\ot c_3,&&\varepsilon_{-1,-1,2}=c_6\ot c_7+c_7\ot c_6;
\end{align*}
the subindices of $\varepsilon_{i,j}$ refer to the degree of $c_k$, see the appendix. We will see that the direct summand $\fL(\varepsilon)$ in the proposition is the following submodule.

\begin{lema}\label{le:Le en sigma - sigma -}
Let $\varepsilon_{-2}=-\zeta^2\varepsilon_{-1,-1,1}+\varepsilon_{-1,-1,2}+(1-\zeta^2)\varepsilon_{0,-2}-(1-\zeta^2)\varepsilon_{-2,0}$. Then the submodule generated by $\varepsilon_{-2}$ is isomorphic to $\fL(\varepsilon)$.
\end{lema}

\begin{proof}
By explicit computations using the appendix, $\xij{12}\varepsilon_{-2}=0=\yij{12}\varepsilon_{-2}$.
\end{proof}

On the other hand, the weight $\varepsilon$ of $\fA$ in degree $-2$ will be
\begin{align*}
\varepsilon'_{-2}&=18\zeta\varepsilon_{-1,-1,1}-6\zeta\varepsilon_{-1,-1,2}+6\varepsilon_{-2,0}+6\varepsilon_{0,-2}.
\end{align*}
The socle of $\fA$ will be generated by
\begin{itemize}
 \item $\fs=(\zeta c_7-c_5)\ot c_8-c_{10}\ot(\zeta c_7-c_5)
+\zeta^2(c_6-c_4)\ot c_{10}-\zeta^2c_8\ot(c_6-c_4)$.
\end{itemize}
Let $S$ be the $\DSn$-module generated by $\fs$.

\begin{lema}\label{le:fA' sigma menos sigma menos}
Let $\lambda$ be an homogeneous weight of $\bigl(\fL(\sigma,-)\otimes \fL(\sigma,-)\bigr)(\ell)$ and $\D\lambda$ denote the submodule generated by $\lambda$.
Hence
\begin{enumerate}[label=(\roman*)]
 \item $\D S\simeq\fL(\sigma,-)$ with highest-weight $S\simeq(\sigma,-)$.
 \smallskip
 \item If $\lambda\in\{\varepsilon,(e,\rho),(\tau,0)\}$ and $\ell=0$, then $\lambda$ is a highest-weight and $\D\lambda$ is an extension of $\fL(\lambda)$ by $\D S$.
 \smallskip
 \item\label{item:fA' -4} If $\lambda\in\{\varepsilon,(e,\rho),(\tau,0)\}$ and $\ell=-4$, then $\lambda$ is a lowest-weight and $\D\lambda$ is an extension of $\fL(\lambda)$ by $\D S$.
 \smallskip
 \item If $\lambda=\ku\varepsilon'_{-2}$, then $\D\lambda$ is an extension of $\fL(\lambda)$ by $\D S$.
 \smallskip
 \item\label{item:fA'} Let $\fA'$ be the sum of all above submodules. Then $\fA'$ is indecomposable with simple socle $\D S$.
\end{enumerate}
\end{lema}

\begin{proof}
By the fusion rules \cite[\S 2.5.4]{PV2}, the homogeneous weight $\varepsilon$ of degre zero is spanned by $\varepsilon_0=c_8\ot c_8+c_9\ot c_9+c_{10}\ot c_{10}$. Clearly, this is a 
highest-weight. Then $\D\varepsilon_0$ is a quotient of the Verma module $\fM(\varepsilon)$ via the morphism $\pi : \fM(\varepsilon) \longrightarrow \D\varepsilon_0$, $\pi(x\otimes 1)=x 
\cdot\varepsilon_0$ for all $x\in\BV(V)$. Using the appendix, we see that $(1-\zeta)\xij{23}\cdot\varepsilon_0=\fs$ and $\xtop\cdot\varepsilon_0=0$. By inspecting the quotients of 
$\fM(\varepsilon)$ in \cite[Theorem 8]{PV2}, we deduce (i) and (ii) for $\lambda=\varepsilon$.

The elements $t=c_8\ot c_8+\zeta^2c_9\ot c_9+\zeta c_{10}\ot c_{10}$ and $u=c_8\ot c_9+c_{10}\ot c_8 + c_9\ot c_{10}$ generate the highest-weights $(e,\rho)$ and $(\tau,0)$ in degree zero, 
respectively; again, this holds by the fusion rules \cite[\S 2.5.4]{PV2}. Then $\D t$ and $\D u$ are quotient of the Verma modules $\fM(e,\rho)$ and $\fM(\tau,0)$, respectively. We finish 
the proof of (ii) by noting that $V\cdot t$ and $V\cdot u$ are contained in $S$. In fact,
\begin{align*}
\fs=\frac{\zeta-1}{\zeta^2}\bigl(1-(23)\bigr)\xij{23}\cdot t=(\zeta-1)\bigl(1-(23)\bigr)\xij{12}\cdot u.
\end{align*}

(iii) The homogeneous weights $\varepsilon$, $(e,\rho)$ and $(\tau,0)$ of degree $-4$ are generated by 
\begin{align*}
\varepsilon_{-4}&=c_1\ot c_1+c_2\ot c_2 + c_3\ot c_3,\\
v&=c_1\ot c_1+\zeta^2c_2\ot c_2 +\zeta c_3\ot c_3\quad\mbox{and}\\
w&=c_1\ot c_2+c_3\ot c_1 + c_2\ot c_3,
\end{align*}
respectively, cf. \cite[\S 2.5.4]{PV2}. Clearly, these are lowest-weights and we have that 
\begin{align*}
\bigl(1-(12)\bigr)\yij{12}\cdot\varepsilon_{-4}=\bigl(1-(12)\bigr)\yij{12}\cdot v=\bigl(1-(12)\bigr)\yij{13}\cdot w.
\end{align*}
Moreover, this element is $\xij{13}\xij{12}\xij{23}\cdot\varepsilon_0$ which generates the lowest-weight of $\D S$ thanks to \cite[Theorem 8]{PV2}. This means that 
$\oV\cdot\varepsilon_{-4}$, $\oV\cdot v$ and $\oV\cdot w$ are contained in $\D S$. Hence (iii) follows from Lemma \ref{le:coVerma}.

\

(iv) We have that 
\begin{align*}
\xij{12}\cdot\varepsilon'_{-2}=(1-\zeta)\xij{13}\xij{12}\xij{23}\cdot\varepsilon_0\quad\mbox{and}\quad
\yij{12}\cdot\varepsilon'_{-2}=(13)\fs
\end{align*}
belong in $\D S$. Therefore $\D \varepsilon'_{-2}=\ku\varepsilon'_{-2}\oplus\D S$ as $\DSn$-modules and (iv) follows.

(v) is a direct consequence of the above.
\end{proof}

By \eqref{eq:fL sigma menos x fL sigma menos} and Remark \ref{obs:sobre los subcocientes}, there is a graded submodule $\fN$ such that 
$$\fL(\sigma,-)\ot\fL(\sigma,-)\simeq\fL(\tau,1)\oplus\fL(\tau,2)\oplus\fN.$$
Notice that $\ku\varepsilon_{-2}$ and $\fA'$ are submodules of $\fN$ such that $\ku\varepsilon_{-2}\cap\fA'=0$ and
$\chgr\fN=\chgr\fA'+t^{-2}\chgr\fL(\varepsilon)+t^{-1}\chgr\fL(\sigma,-)$. 

\begin{lema}\label{le:fA}
Let $\lambda=(\sigma,-)$ be an homogeneous weight of degree $-1$ or $-3$ which is not contained in $\D S$. Hence $\D\lambda\supset\fA'$ and $\D\lambda/\fA'\simeq\fL(\sigma,-)$.
\end{lema}

\begin{proof}
Since $\fL(\sigma,-)$, $\fL(\tau,1)$ and $\fL(\tau,2)$ are self-dual, so is $\fN$. Moreover, as graded modules $\fN\simeq\fN^*[-4]$.

If $\lambda$ is of degree $-1$, then the space of weights $(\sigma,-)$ in $\fN(-1)$ is $N=\lambda\oplus S$. We claim that $\oV\cdot N=\fN(0)=\fA'(0)$. In fact, let $\mu$ be a weight 
of $\fN(0)$ and $\mu^*\subset(\fN(0))^*$ the dual space of $\mu$. We see that
\begin{align*}
\langle\mu^*,\oV\cdot N\rangle=\langle\mu^*,\oV\DSn\cdot N\rangle=\langle\cS(\oV\DSn)\cdot\mu^*,N\rangle=\langle\oV\cdot\mu^*,N\rangle\neq0, 
\end{align*}
and it is non-zero because $\fN\simeq\fN^*[-4]$ and Lemma \ref{le:fA' sigma menos sigma menos} \ref{item:fA' -4}.

In a similar way, we can show that $V\cdot\widetilde{N}=\fN(-4)=\fA'(-4)$ where $\widetilde{N}$ is the space of weights $(\sigma,-)$ in $\fN(-3)$. Also, we can show that $V\cdot N$ has 
a weight $\mu_1\simeq\varepsilon$ and $\oV\cdot\widetilde{N}$ has a weight $\mu_2\simeq\varepsilon$, both weights are of degree $-2$. 

We claim that $\mu_1=\mu_2$. Indeed, the space of weights 
$\varepsilon$ of $\D\lambda/\D\fA'(0)$ is $\mu_1+\mu_2+\ku\varepsilon_{-4}$ where $\ku\varepsilon_{-4}$ is the trivial weight of $\fA'(-4)$. On the other hand, $(\sigma,-)$ is a 
highest-weight generating $\D\lambda/\D\fA'(0)$ and hence $\D\lambda/\D\fA'(0)$ is a quotient of $\fM(\sigma,-)$. As $\fM(\sigma,-)$ has only two copies of $\varepsilon$ we deduce that 
$\mu_1=\mu_2$.

Finally, the element
\begin{align*}
z=3(c_4\ot c_2+c_5\ot c_3)+2(\zeta-1)(c_3\ot c_6+c_2\ot c_7)+(4\zeta^2-\zeta)(c_2\ot c_5+c_3\ot c_4)
\end{align*}
belongs in a weight $(\sigma,-)$ in $\fA'(-3)$ by the fusion rules. Moreover, we have that 
$$\varepsilon'_{-2}=\bigr(1+(13)+(23)\bigl)\yij{12}\cdot z.$$
Therefore $\ku\varepsilon_{-2}'=\mu_1=\mu_2$. This finishes the proof.
\end{proof}

\begin{proof}[Proof of Proposition \ref{prop: sigma menos por sigma menos}]
Let $\lambda$ be as in Lemma \ref{le:fA}. Then $\fA=\D\lambda$ satisfies the properties of the statement by Lemmas \ref{le:Le en sigma - sigma -}, \ref{le:fA' sigma menos sigma menos} 
and \ref{le:fA}.
\end{proof}

\subsection{The case \texorpdfstring{$\fL(e,\rho)\otimes \fL(e,\rho)$}{s}}

As in \eqref{eq:fL tau 0 x fL e rho} we can see that
\begin{align}\label{eq:fL e rho x fL e rho}
\chgr&\left(\fL(e,\rho)\otimes \fL(e, \rho)\right)=\\
\notag=&\chgr\fL(e,-)+(1+t^{-2}+t^{-4})\chgr\fL(\varepsilon)+(1+t^{-2})\chgr\fL(e, \rho)+2t^{-1}\fL(\sigma,-).
\end{align}
Therefore $\fL(e,-)$ is a graded direct summand of $\fL(e,\rho)\otimes \fL(e,\rho)$. 

\begin{prop}\label{prop:fB}
Let $\fB$ be a graded complement of $\fL(e,-)$. Then $\fB$ is indecomposable and
$$\fL(e,\rho)\otimes \fL(e,\rho)\simeq\fL(e,-)\oplus\fB$$
as graded modules. Moreover,
\begin{align*}
\operatorname{soc}\fB&=\fH\simeq t^{-1}\fL(\sigma,-),\\
\soco^2\fB/\soco\fB&\simeq (1+t^{-2}+t^{-4})\fL(\varepsilon)\oplus(1+t^{-2})\fL(e, \rho),\\
\soco^3\fB/\soco^2\fB&\simeq t^{-1}\fL(\sigma,-),\\
\soco^3\fB&=\fB.
\end{align*}

\end{prop}

\begin{proof}
By Remark \ref{obs:sobre los subcocientes}, there exists $\fB$ such that $\fL(e,\rho)\otimes \fL(e,\rho)\simeq\fL(e,-)\oplus\fB$. We will show in Lemma \ref{le:submod fL e rho fL e rho} that 
such a $\fB$ satisfies the required properties.
\end{proof}

Figure \ref{fig:fB} helps the reader to visualize the module $\fB$ and to follow the proof of the next lemma.

\begin{figure}[h]
\begin{center}
\includegraphics{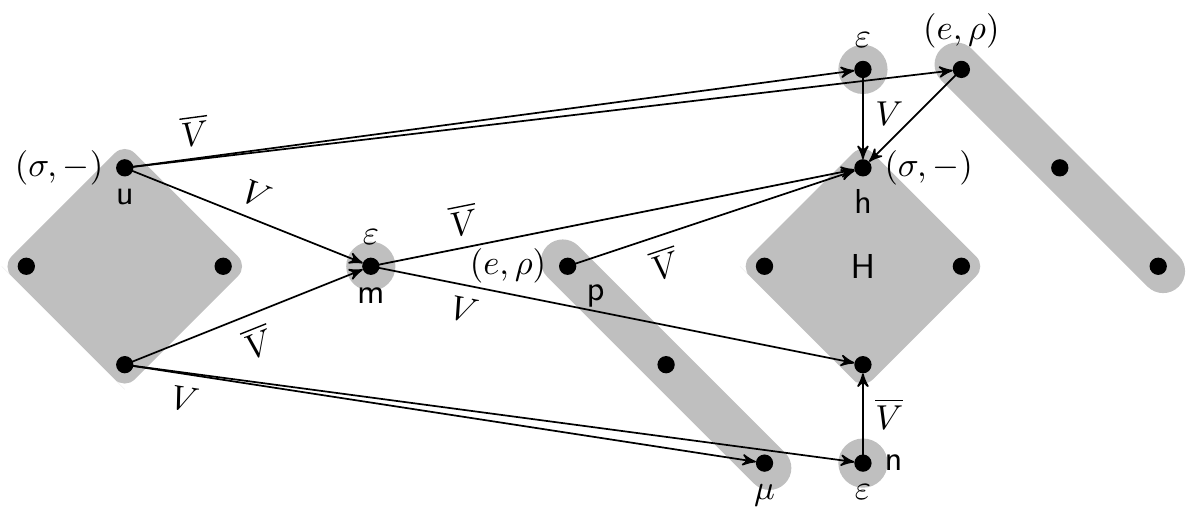}
\caption{The dots represent the weights of $\fB$. Each shadow region correspond to a composition factor whose highest-weight is in the top. The actions of $V$ and $\oV$ are illustrated by 
the arrows.}\label{fig:fB}
\end{center}
\end{figure}

We define the elements $\fh,\fh'\in\fB(-1)$ by
\begin{align*}
\fh=b_4\otimes (b_7-b_6)-(b_7-b_6)\otimes b_4\quad\mbox{and}\quad\fh'=b_4\ot b_7-b_4\ot b_6.
\end{align*}
Using the fusion rule \cite[(15)]{PV2} we obtain that $\DSn\fh\simeq\DSn\fh'\simeq(\sigma,-)$. Moreover, the space of weights $(\sigma,-)$ of $\fB(-1)$ is $\DSn\fh\oplus\DSn\fh'$ 
by \eqref{eq:fL e rho x fL e rho}.

Let $\fH$ be the submodule generated by $\fh$. It is a highest-weight module since $\yij{12}\fh=0$, which is a straightforward computation using the appendix.

\begin{lema}\label{le:submod fL e rho fL e rho}
Let $\lambda$ be a homogeneous weight of $\fB(\ell)$ and $\D\lambda$ denote the submodule generated by $\lambda$.
\begin{enumerate}[label=(\roman*)]
 \item\label{item: H le:submod fL e rho fL e rho}  $\fH\simeq\fL(\sigma,-)$.
 \smallskip
 \item\label{item: e 0 le:submod fL e rho fL e rho} If $\lambda=\varepsilon$ and $\ell=0$, then $\lambda$ is a highest-weight and $\D\lambda$ is an extension of $\fL(\varepsilon)$ by $\fH$.
  \smallskip
 \item\label{item: e -2 le:submod fL e rho fL e rho} If $\lambda=\varepsilon$ and $\ell=-2$, then $\D\lambda$ is an extension of $\fL(\varepsilon)$ by $\fH$.
 \smallskip
 \item\label{item: e -4 le:submod fL e rho fL e rho} If $\lambda=\varepsilon$ and $\ell=-4$, then $\lambda$ is a lowest-weight and $\D\lambda$ is an extension of $\fL(\varepsilon)$ by $\fH$.
 \smallskip
 \item\label{item: e rho 0 le:submod fL e rho fL e rho} If $\lambda=(e,\rho)$ and $\ell=0$, then $\lambda$ is a highest-weight and $\D\lambda$ is an extension of $\fL(e,\rho)$ by $\fH$.
 \smallskip
 \item\label{item: e rho -2 le:submod fL e rho fL e rho} If $\lambda=(e,\rho)$ and $\ell=-2$, then $\D\lambda$ is is an extension of $\fL(e,\rho)$ by $\fH$.
 \smallskip
 \item\label{item: sigma - le:submod fL e rho fL e rho} If $\lambda=(\sigma,-)\neq\DSn\fh$ and $\ell=-1$, then $\fB=\D\lambda$.
\end{enumerate}
\end{lema}

\begin{proof}
Assume that $\lambda=\varepsilon$ and $\ell=0$. A basis of $\lambda$ is $b_6\ot b_7+b_7\ot b_6$ by \cite[\S 2.5.4]{PV2}. Clearly, $\lambda$ is a highest-weight. Then $\D\lambda$ is a quotient of the Verma module $\fM(\varepsilon)$ via the morphism $\pi : \fM(\varepsilon) \longrightarrow \D\lambda$, $\pi(x\otimes 1)=x \cdot(b_6\ot b_7+b_7\ot b_6)$ for all $x\in\BV(V)$. Recall the quotients of $\fM(\varepsilon)$ from \cite[Theorem 8]{PV2}. Since $\xij{12}\cdot(b_6\ot b_7+b_7\ot b_6)=\fh$ and $\xtop\cdot(b_6\ot b_7+b_7\ot b_6)=0$, $\D\lambda$ fits in an exact sequence $\fL(\sigma,-)\longrightarrow\D\lambda\longrightarrow\fL(\varepsilon)$ by \cite[Theorem 10]{PV2}. Since $\fH$ is a submodule of $\D\lambda$ we deduce that $\fH\simeq\fL(\sigma,-)$ and \ref{item: H le:submod fL e rho fL e rho} and \ref{item: e 0 le:submod fL e rho fL e rho} follow.

\

In case \ref{item: e -2 le:submod fL e rho fL e rho}, $\fm=b_3\ot b_3+b_4\ot b_4+b_5\ot b_5$ is a basis of $\lambda$, cf. \cite[\S 2.5.4]{PV2}. Then we see that $\yij{12}\cdot\fm=\fh$ and $\xij{12}\cdot\fm=(13)\xij{13}\xij{23}\cdot\fh$.

\

For \ref{item: e -4 le:submod fL e rho fL e rho}, $\fn=b_1\ot b_2+b_2\ot b_1$ is a basis of $\lambda$ and we have that $\xij{12}\cdot\fn=0$ and $\yij{12}\cdot\fn=(13)\xij{13}\xij{23}\cdot\fh$.

\

\ref{item: e rho 0 le:submod fL e rho fL e rho} A basis of $\lambda$ is formed by $b_6\ot b_6$ and $b_7\ot b_7$. Clearly $\lambda$ is a highest-weight. Then $\D\lambda$ is a quotient of the 
Verma module $\fM(e,\rho)$. Let $\pi : \fM(e,\rho) \longrightarrow \D\lambda$ be the induced morphism, which is analogous to that in the case (i). Since $\fB$ has no composition factors 
isomorphic to $\fL(\tau,0)$ by \eqref{eq:fL e rho x fL e rho}, we deduce that the socle of $\fM(e,\rho)$ is contained in $\ker\pi$, see \cite[Theorem 10]{PV2}. Finally, we have that 
$\pi(\fe_0)=-\fh$, cf. \cite[\S 4.6]{PV2}. 

\ 

\ref{item: e rho -2 le:submod fL e rho fL e rho} In this case $\D\lambda$ has a composition factor isomorphic to $\fL(e,\rho)[-2]$ by Remark \ref{obs:sobre 
los subcocientes}. Since $\chgr\fL(e,\rho)\simeq(e,\rho)+t^{-1}(\sigma,-)+t^{-2}(\tau,0)$, $\BV^2(V)\lambda$ contains the unique weight $\mu=(\tau,0)$ of $\fB(-4)$ and $\D\mu=\D\lambda$. 
Notice that $\mu$ is a lowest-weight. Hence, $\D\lambda$ is a quotient of the co-Verma module $\fW(\tau,0)$. We have that $\fp=\zeta^2b_3\ot b_3+b_4\ot b_4+\zeta b_5\ot b_5$ 
belongs in $\lambda$ and $\yij{12}\cdot\fp=\fh$. Then $\D\lambda$ is an extension of $\fL(\lambda)$ by $\fH$ thanks to Lemma \ref{le:coVerma}. 

\

\ref{item: sigma - le:submod fL e rho fL e rho} Let  $\fu=s\fh+t\fh'$ be a generator of $\lambda$ for some $s,t\in\ku$ with $t\neq0$, that is $\DSn\fu=\lambda$. Then, the next elements are 
linearly independent:
\begin{align*}
\yij{12}\cdot\fu&=-t(b_7-b_6)\ot(b_7-b_6),\\ 
(13)\yij{12}\cdot\fu&=-t(\zeta b_7-\zeta^2b_6)\ot(\zeta b_7-\zeta^2b_6),\\ 
(23)\yij{12}\cdot\fu&=-t(\zeta^2 b_7-\zeta b_6)\ot(\zeta^2 b_7-\zeta b_6)
\end{align*}
Hence $\oV\cdot\lambda$ coincides with the $\DSn$-submodule $(e,\rho)\oplus\varepsilon$ contained in $\fB(0)$ by the fusion rules. Therefore the submodules in \ref{item: H le:submod fL e 
rho fL e rho}, \ref{item: e 0 le:submod fL e rho fL e rho} and \ref{item: e rho 0 le:submod fL e rho fL e rho} are contained in $\D\lambda$.

On the other hand, 
\begin{align*}
\xij{23}\xij{12}\xij{13}\cdot\fu=t\fn\quad\mbox{and}\quad\xij{13}\xij{12}\xij{13}\cdot\fu=-2t\,b_1\ot b_1,
\end{align*}
where the second element belongs in the weight $(\tau,0)$ of $\fB(-4)$ by the fusion rules. Hence $\D\lambda$ contains the submodules in \ref{item: e 
-4 le:submod fL e rho fL e rho} and \ref{item: e rho -2 le:submod fL e rho fL e rho}.

Finally,  $\bigl(1+(13)+(23)\bigr)\xij{12}\cdot\fu=-t\fm$. Then $\chgr\D\lambda=\chgr\fB$ and \ref{item: sigma - le:submod fL 
e rho fL e rho} follows. 
\end{proof}

\subsection{The case \texorpdfstring{$\fL(\sigma,-)\otimes \fL(e,\rho)$}{s}}

In $K[t,t^{-1}]$ it holds that
\begin{align}\label{eq:fL sigma menos x fL e rho}
\chgr&\left(\fL(\sigma,-)\otimes \fL(e, \rho)\right)=\chgr\fL(\sigma,+)+2t^{-1}\chgr\fL(\tau,0)+(1+t^{-2})\chgr\fL(\sigma,-).
\end{align}
Therefore $\fL(\sigma,+)$ is a graded direct summand of $\fL(\sigma,-)\otimes \fL(e,\rho)$. 

\begin{prop}\label{prop:fC}
Let $\fC$ be a graded complement of $\fL(\sigma,+)$. Then $\fC$ is indecomposable, $\chgr\fC=2t^{-1}\chgr\fL(\tau,0)+(1+t^{-2})\chgr\fL(\sigma,-)$ and
$$\fL(\sigma,-)\otimes \fL(e,\rho)\simeq\fL(\sigma,+)\oplus\fC$$
as graded modules. Moreover, the socle filtration of $\fC$ satisfies
\begin{align*}
\soco\fC&\simeq t^{-1}\fL(\tau,0),\\
\soco^2\fC/\soco\fC&\simeq(1+t^{-2})\fL(\sigma,-),\\
\soco^3\fC/\soco^2\fC&\simeq t^{-1}\fL(\tau,0),\\
\soco^3\fC&=\fC.
\end{align*}

\end{prop}

\begin{figure}[h]
\begin{center}
\includegraphics{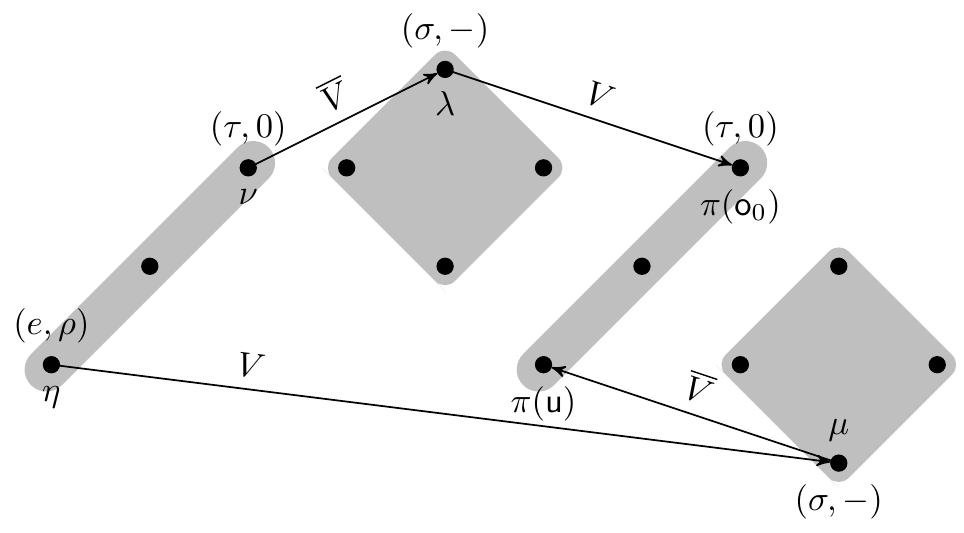}
\caption{The dots represent the weights of $\fC$. Each shaded area corresponds to a composition factor whose highest-weight is in the top. The actions of $V$ and $\oV$ are illustrated by 
the arrows.}\label{fig:fC}
\end{center}
\end{figure}

\begin{proof}
The weight $\lambda=(\sigma,-)$ of $\fC(0)$ is a highest-weight. Then we have a projection $\pi:\fM(\sigma,-)\longrightarrow\D\lambda$ and hence $\D\lambda$ is a quotient of 
$\fM(\sigma,-)$ by one of the submodules given in \cite[Theorem 7]{PV2}. By $\chgr\fC$, we see that either $\D\lambda\simeq\fL(\sigma,-)$ or $\D\lambda$ is an extension of $\fL(\sigma,-)$ 
by $\fL(\tau,0)$.

By the fusion rules \cite[\S 2.5.4]{PV2}, $\ft=c_8\ot(b_6+b_7)$ generates $\lambda$. Let $\fo_0\in\fM(\sigma,-)$ be as in \cite[Lemma 23]{PV2}. Then 
\begin{align*}
0\neq\pi(\fo_0)&=\xij{13}\cdot\ft+\xij{12}\cdot\bigl(c_9\ot(\zeta^2b_6+\zeta b_7)\bigr)+
\xij{23}\cdot\bigl(c_{10}\ot(\zeta b_6+\zeta^2b_7)\bigr)\\
&=(\zeta-\zeta^2)(c_8\ot b_3+c_9\ot b_5+c_{10}\ot b_4)+\frac{3}{1-\zeta}(\zeta c_6\ot b_7-c_4\ot b_6).
\end{align*}
Therefore $\D\lambda$ is an extension of $\fL(\sigma,-)$ by $\fL(\tau,0)$. In particular, this shows that $t^{-1}\fL(\tau,0)\subseteq\soco\fC$.

If $\fu\in\fM(\sigma,-)$ is as in \cite[Lemma 23]{PV2}, we have

\begin{align*}
\pi(\fu)&=\zeta^2\xij{12}\xij{13}\xij{12}(13)\cdot\ft-\zeta\xij{12}\xij{13}\xij{23}(23)\cdot\ft+\xij{13}\xij{12}\xij{23}\cdot\ft\\
&=3\bigl(\zeta^2c_2\ot b_3+\zeta c_3\ot b_5-c_4\ot b_2+c_1\ot b_4+c_7\ot b_1\bigr).
\end{align*}

\

On the other hand, the lowest-weight $\mu=(\sigma,-)\subset\fC(-4)$ is generated by $\ft'=c_3\ot b_1+c_2\ot b_2$, cf. \cite[\S 2.5.4]{PV2}. We have that
\begin{align*}
\yij{12}\cdot\ft'-\zeta^2\yij{23}(13)\cdot\ft'-\zeta\yij{13}(23)\cdot\ft'&=\\
=(1-\zeta)c_2\ot b_3+(\zeta^2-1)c_3\ot b_5-&\frac{3\zeta^2}{\zeta-1}c_4\ot b_2+(\zeta-\zeta^2)c_1\ot b_4+\frac{3\zeta^2}{\zeta-1}c_7\ot b_1\\
&=\frac{\zeta^2}{\zeta-1}\pi(\fu).
\end{align*}
Therefore $\D\mu$ is an extension of $\fL(\sigma,-)$ by $\fL(\tau,0)$ thanks to Lemma \ref{le:coVerma}, and $\D\mu\cap\D\lambda\simeq\fL(\tau,0)$.

\

Let $N$ denote the space of weights $(\tau,0)$ contained in $\bigl(\fL(\sigma,-)\otimes \fL(e, \rho)\bigr)(-1)$. By the fusion rules \cite[\S 2.5.4]{PV2}, $N$ is generated by
\begin{align*}
\{c_4\ot b_6,\quad c_6\ot b_7,\quad c_8\ot b_3+c_9\ot b_5+c_{10}\ot b_4\}.
\end{align*}
Hence $\oV\cdot N\subset\lambda$. In fact,
\begin{align*}
\zeta\ft=&\bigl(1-(12)\bigr)\yij{13}\cdot(c_4\ot b_6),\\
\zeta^2\ft=&\bigl(1-(12)\bigr)\yij{13}\cdot(c_6\ot b_7)\quad\mbox{and}\\
(\zeta-\zeta^2)\ft=&\bigl(1-(12)\bigr)\yij{13}\cdot(c_8\ot b_3+c_9\ot b_5+c_{10}\ot b_4).
\end{align*}
In particular, if $\nu\simeq(\tau,0)$ is a weight of $\fC(-1)$ which is different from $\DSn\cdot\pi(\fo_0)$, then $\oV\cdot\nu=\lambda$. Hence $\D\nu$ contains $\D\lambda$.

Let $\eta\simeq(e,\rho)$ be a weight of $\D\nu$ which is different from $\DSn\cdot\pi(\fu)$. We claim that $\oV\cdot\eta=\mu$. Otherwise, $\eta$ should be a lowest-weight because of 
$\chgr\fC$. Hence $\D\nu$ is a quotient of $\fW(e,\rho)$ with two composition factors isomorphic to $\chgr\fL(\tau,0)$. However, this can not happen 
by Lemma \ref{le:coVerma} and our claim follows.

Therefore $\D\nu=\fC$ because $\chgr\D\nu=\chgr\fC$.
\end{proof}

\subsection{The remainder cases}

The functor $\fL(\varepsilon)\ot-$ is the identity and $\fM\ot\fN\simeq\fN\ot\fM$ because $\D$ is quasitriangular. Thus, we finish the description of the tensor product between 
non-projective simple modules with the next proposition.

\begin{prop}\label{prop:remainder}
We have that
\begin{align*}
\fL(\tau,0)\ot\fL(\tau,0)&\simeq\fL(e,-)\oplus\fB^*,\\
\fL(\sigma,-)\ot\fL(\tau,0)&\simeq\fL(\sigma,+)\oplus\fC^*.
\end{align*}
\end{prop}

\begin{proof}
It follows from dualizing the isomorphisms of Propositions \ref{prop:fB} and \ref{prop:fC}.
\end{proof}

\section{The projective modules}\label{sec:proyectivos}

We denote by $\fP(\lambda)$ the projective cover of $\fL(\lambda)$ for all $\lambda\in\Lambda$. Since $\D$ is symmetric \cite{MR1435369}, $\fP(\lambda)$ also is the injective  hull of $\fL(\lambda)$. 

Up to shifts, $\fP(\lambda)$ admits a unique $\Z$-grading \cite{MR659212}. We fix one such that $\lambda$ is a homogeneous weight of degree $0$ generating $\fP(\lambda)$. Thus, $\fP(\lambda)$ also is the projective cover and the injective hull of $\fL(\lambda)$ as a graded module, cf. \cite[Lemma 8]{vay-proj}.

\

Let $R^\bullet_{proj}$ denote the Grothendieck ring of the subcategory of projective modules. The sets $\{\chgr\fP(\lambda)\mid\lambda\in\Lambda\}$, $\{\chgr\fM(\lambda)\mid\lambda\in\Lambda\}$ and $\{\chgr\fW(\lambda)\mid\lambda\in\Lambda\}$ are $\Z[t,t^{-1}]$-bases of $R^\bullet_{proj}$ \cite[Remark 3]{vay-proj}. Then, for every graded projective module $\fP$, there are polynomials $p_{\fP,\fP(\lambda)}$, $p_{\fP,\fM(\lambda)}$ and $p_{\fP,\fW(\lambda)}$ in $\Z[t,t^{-1}]$ satisfying the following properties.
\begin{align}
\label{eq:proj D mod}
\chgr\fP&=\sum_{\lambda\in\Lambda}p_{\fP,\fP(\lambda)}\chgr\fP(\lambda)\Longleftrightarrow
\fP\simeq\oplus_{\lambda\in\Lambda}p_{\fP,\fP(\lambda)}\fP(\lambda)\quad\mbox{as graded modules.}\\
\label{eq:proj D-0 mod}
\chgr\fP=&\sum_{\lambda\in\Lambda}p_{\fP,\fM(\lambda)}\chgr\fM(\lambda)\Longleftrightarrow
\fP\simeq\oplus_{\lambda\in\Lambda}p_{\fP,\fM(\lambda)}\fM(\lambda)\quad\mbox{as graded $\D^{\leq0}$-modules.}\\
\label{eq:proj D+0 mod}
\chgr\fP=&\sum_{\lambda\in\Lambda}p_{\fP,\fW(\lambda)}\chgr\fW(\lambda)\Longleftrightarrow
\fP\simeq\oplus_{\lambda\in\Lambda}p_{\fP,\fW(\lambda)}\fW(\lambda)\quad\mbox{as graded $\D^{\geq0}$-modules.}
\end{align}

The graded BGG Reciprocity \cite[Corollary 12 and Theorem 20]{vay-proj} states that 
\begin{align}\label{eq:polinomios de verma y coverma}
p_{\fP(\mu),\fM(\lambda)}=\overline{p_{\fM(\lambda),\fL(\mu)}}=t^4\,p_{\fP(\mu),\fW(\lambda)}.
\end{align}
for all $\mu,\lambda\in\Lambda$. Therefore,
\begin{align*}
 \chgr\fP(\varepsilon)&=(1+t^{4})\chgr\fM(\varepsilon)+(t+t^3)\chgr\fM(\sigma,-),\\
 \noalign{\smallskip}
 \chgr\fP(e,\rho)&=\chgr\fM(e,\rho)+t\chgr\fL(\sigma,-)+t^{2}\chgr\fM(\tau,0),\\
 \noalign{\smallskip}
 \chgr\fP(\sigma,-)&=(1+t^2)\chgr\fL(\sigma,-)+t\chgr\fM(\varepsilon)+t\chgr\fM(e,\rho)+t\chgr\fM(\tau,0),\\
 \noalign{\smallskip}
 \chgr\fP(\tau,0)&=\chgr\fM(\tau,0)+t\chgr\fM(\sigma,-)+t^{2}\chgr\fM(e,\rho),\\
 \noalign{\smallskip}
 \chgr\fP(\lambda)&=\chgr\fM(\lambda),\quad\forall\,\lambda\in\Lsp.
\end{align*}

\

We give more information on the structure of the indecomposable projective modules using \cite[Remark 4]{vay-proj}. In the following, if $\fM(\lambda)[\ell]$ is a graded shift of a Verma 
module, we shall denote its highest-weight by $1\ot\lambda[\ell]$. We will omit $\ell$ if it is zero.

\begin{prop}\label{prop:P sigma menos}
As graded $\D^{\leq0}$-modules,
\begin{align*}
\fP(\sigma,-)=\fM(\sigma,-)[2]\oplus\fM(\varepsilon)[1]\oplus\fM(e,\rho)[1]\oplus\fM(\tau,0)[1]\oplus\fM(\sigma,-). 
\end{align*}
The action of $\oV$ satisfies:
\begin{align*}
\oV\cdot\left(1\ot(\sigma,-)[2]\right)&=0,\\
\oV\cdot\left(1\ot\varepsilon[1]\right)&=1\ot(\sigma,-)[2],\\
\oV\cdot\left(1\ot(e,\rho)[1]\right)&=1\ot(\sigma,-)[2],\\
\oV\cdot\left(1\ot(\tau,0)[1]\right)&=1\ot(\sigma,-)[2].
\end{align*}
Moreover, the projection of $\oV\cdot\left(1\ot(\sigma,-)\right)$ over $\fM(\lambda)[1]$ is equal to $(1\ot\lambda)[1]$ for all $\lambda\in\{\varepsilon,(e,\rho),(\tau,0)\}$.

Therefore
\begin{enumerate}[label=(\roman*)]
\item\label{item:submod 2 prop:P sigma menos} The submodule generated by $1\ot(\sigma,-)[2]$ is isomorphic to $\fM(\sigma,-)[2]$. 
 \smallskip
 \item\label{item:submod 1 prop:P sigma menos} The submodule generated by $1\ot\lambda[1]$ is equal to $\fM(\sigma,-)[2]\oplus\fM(\lambda)[1]$ as graded $\D^{\leq0}$-module for all $\lambda\in\{\varepsilon,(e,\rho),(\tau,0)\}$.
 \smallskip
 \item\label{item:generator prop:P sigma menos} $\fP(\sigma,-)$ is generated by the homogeneous weight $1\ot(\sigma,-)$ of degree $0$.
 \smallskip
 \item\label{item:standard filt prop:P sigma menos} The following are standard filtrations of $\fP(\sigma,-)$
 \begin{align*}
 \fM(\sigma,-)[2]\subset&\quad\D\cdot\left(1\ot\lambda_1[1]\right)\quad\subset\quad\D\cdot\left(1\ot\lambda_1[1]\right)+\D\cdot\left(1\ot\lambda_2[1]\right)\\
 &\subset\quad\D\cdot\left(1\ot\lambda_1[1]\right)+\D\cdot\left(1\ot\lambda_2[1]\right)+\D\cdot\left(1\ot\lambda_3[1]\right)
 \quad\subset\quad\fP(\sigma,-)
 \end{align*}
where $\{\lambda_1,\lambda_2,\lambda_3\}=\{\varepsilon,(e,\rho),(\tau,0)\}$. 
\end{enumerate}
\end{prop}

\begin{proof}
The structure of $\D^{\leq0}$-module of $\fP(\sigma,-)$ follows by \eqref{eq:proj D-0 mod}. A direct consequence of this isomorphism is that $\fM(\lambda)[\ell_\lambda]$ is a graded submodule of $\fP(\sigma,-)$ if $1\ot\lambda[\ell_\lambda]$ is a highest-weight. But $\fP(\sigma,-)$ has only one Verma submodule because its socle is simple. Then we see that such a Verma module is $\fM(\sigma,-)[2]$.

To calculate the $\oV$-actions, we shall use the grading on $\fP(\sigma,-)$ which ensures that $\oV\cdot(1\ot\lambda[\ell_\lambda])\subseteq\fP(\sigma,-)(\ell_\lambda+1)$. Then the action of $\oV$ on $\left(1\ot(\sigma,-)\right)[2]$ is zero because $\fP(\sigma,-)(3)=0$. This shows \ref{item:submod 2 prop:P sigma menos}.

\smallskip

By the graded character, $\fP(\sigma,-)(2)=1\ot(\sigma,-)[2]$ and then $0\neq\oV\cdot\left(1\ot\lambda[1]\right)\subseteq1\ot(\sigma,-)[2]$. Hence the equality holds, because $1\ot(\sigma,-)$ is a weight, and \ref{item:submod 1 prop:P sigma menos} follows.

\smallskip

We now analyze the action on $1\ot(\sigma,-)$. We have that
$$
\oV\cdot\left(1\ot(\sigma,-)\right)\subset\left(1\ot\varepsilon\right)[1]\oplus\left(1\ot(e,\rho)\right)[1]\oplus\left(1\ot(\tau,0)\right)[1]\oplus\fM(\sigma,-)[2](1).
$$
If the projection of $\oV\cdot\left(1\ot(\sigma,-)\right)$ over $1\ot\lambda[1]$ is zero for some $\lambda\in\{\varepsilon,(e,\rho),(\tau,0)\}$, then the submodule $\fN$ generated by $1\ot(\sigma,-)$ satisfies $\fP(\sigma,-)/\fN\simeq\fM(\lambda)[1]$ by \ref{item:submod 1 prop:P sigma menos}. But this is not possible since $\fP(\sigma,-)$ has simple head. Hence the projection is equal to $1\ot\lambda[1]$ because it is a weight. In particular, we see that \ref{item:generator prop:P sigma menos} holds.

\smallskip

The filtrations in \ref{item:standard filt prop:P sigma menos} are standard by \ref{item:submod 1 prop:P sigma menos} and \ref{item:generator prop:P sigma menos}.
\end{proof}

The demonstrations of the next results are analogous to Proposition \ref{prop:P sigma menos}. 

\begin{prop}\label{prop:P e mas}
As graded $\D^{\leq0}$-modules,
\begin{align*}
\fP(\varepsilon)=\fM(\varepsilon)[4]\oplus\fM(\sigma,-)[3]\oplus\fM(\sigma,-)[1]\oplus\fM(\varepsilon). 
\end{align*}
The action of $\oV$ satisfies:
\begin{align*}
\oV\cdot\left(1\ot\varepsilon[4]\right)&=0,\\
\oV\cdot\left(1\ot(\sigma,-)[3]\right)&=1\ot\varepsilon[4].
\end{align*}
Moreover, the projection of $\oV\cdot\left(1\ot\varepsilon\right)$ over $\fM(\sigma,-)[1]$ is equal to $1\ot(\sigma,-)[1]$.

Therefore
\begin{enumerate}[label=(\roman*)]
\item\label{item:submod 4 prop:P e mas} $\D\cdot(1\ot\varepsilon[4])\simeq\fM(\varepsilon)[4]$. 
\smallskip
\item\label{item:submod 1 prop:P e mas} $\D\cdot\left(1\ot(\sigma,-)[3]\right)=\fM(\varepsilon)[4]\oplus\fM(\sigma,-)[3]$ as graded $\D^{\leq0}$-modules.
 \smallskip
 \item\label{item:submod 3 prop:P e mas} $\D\cdot\left(1\ot(\sigma,-)[1]\right)+\D\cdot\left(1\ot(\sigma,-)[3]\right)=\fM(\varepsilon)[4]\oplus\fM(\sigma,-)[3]\oplus\fM(\sigma,-)[1]$ as graded $\D^{\leq0}$-modules.
 \smallskip
 \item\label{item:generator prop:P e mas} $\fP(\varepsilon)=\D\cdot(1\ot\varepsilon)$.
 \smallskip
 \item\label{item:standard filt prop:P e mas} The following is a standard filtration of $\fP(\varepsilon)$
 \begin{align*}
 \fM(\varepsilon)[4]\,\subset\,\D\cdot\left(1\ot(\sigma,-)[3]\right)
 \,\subset\,\D\cdot\left(1\ot(\sigma,-)[1]\right)+\D\cdot\left(1\ot(\sigma,-)[3]\right)\,\subset\,\fP(\varepsilon).
 \end{align*}
 \end{enumerate}
\end{prop}

\begin{proof}
The equality for the action of $\oV$ over $1\ot\varepsilon[4]$ and $1\ot(\sigma,-)[3]$ is a direct consequence of the grading. Hence, we can deduce that 
the projection of $\oV\cdot\left(1\ot\varepsilon\right)$ over $\fM(\sigma,-)[1]$ is equal to $1\ot(\sigma,-)[1]$ arguing as in the above proposition. For \ref{item:submod 3 prop:P e mas} note that $\oV\cdot\left(1\ot(\sigma,-)[1]\right)\subset\fP(\varepsilon)(2)\subset\fM(\varepsilon)[4]\oplus\fM(\sigma,-)[3]$.
\end{proof}

\begin{prop}\label{prop:P e rho}
As graded $\D^{\leq0}$-modules,
\begin{align*}
\fP(e,\rho)=\fM(\tau,0)[2]\oplus\fM(\sigma,-)[1]\oplus\fM(e,\rho). 
\end{align*}
The action of $\oV$ satisfies:
\begin{align*}
\oV\cdot\left(1\ot(\tau,0)[2]\right)&=0,\\
\oV\cdot\left(1\ot(\sigma,-)[1]\right)&=1\ot(\tau,0)[2]. 
\end{align*}
Moreover, the projection of $\oV\cdot\left(1\ot(e,\rho)\right)$ over $\fM(\sigma,-)[1]$ is equal to $1\ot(\sigma,-)[1]$.

Therefore
\begin{enumerate}[label=(\roman*)]
\item $\D\cdot\left(1\ot(\tau,0)[2]\right)\simeq\fM(\tau,0)[2]$.
\smallskip
 \item $\D\cdot\left(1\ot(\sigma,-)[1]\right)=\fM(\tau,0)[2]\oplus\fM(\sigma,-)[1]$ as graded $\D^{\leq0}$-modules.
 \smallskip
\item\label{item:generator prop:P e rho} $\fP(e,\rho)=\D\cdot1\ot(e,\rho)$.
 \smallskip
 \item\label{item:standard filt prop:P e rho} The following is a standard filtration of $\fP(e,\rho)$
 \begin{align*}
 \fM(\tau,0)[2]\,\subset\,\D\cdot\left(1\ot(\sigma,-)[1]\right)\,\subset\,\fP(e,\rho).
 \end{align*}
 \qed
 \end{enumerate}
\end{prop}

\begin{prop}\label{prop:P tau zero}
As graded $\D^{\leq0}$-modules,
\begin{align*}
\fP(\tau,0)=\fM(e,\rho)[2]\oplus\fM(\sigma,-)[1]\oplus\fM(\tau,0). 
\end{align*}
The action of $\oV$ satisfies:
\begin{align*}
\oV\cdot\left(1\ot(e,\rho)[2]\right)&=0,\\
\oV\cdot\left(1\ot(\sigma,-)[1]\right)&=\left(1\ot(e,\rho)[2]\right). 
\end{align*}
Moreover, the projection of $\oV\cdot\left(1\ot(\tau,0)\right)$ over $\fM(\sigma,-)[1]$ is equal to $1\ot(\sigma,-)[1]$.

Therefore
\begin{enumerate}[label=(\roman*)]
\item $\D\cdot\left(1\ot(e,\rho)[2]\right)\simeq\fM(e,\rho)[2]$.
\smallskip
\item $\D\cdot\left(1\ot(\sigma,-)[1]\right)=\fM(e,\rho)[2]\oplus\fM(\sigma,-)[1]$ as graded $\D^{\leq0}$-modules.
 \smallskip
 \item\label{item:generator prop:P tau zero} $\fP(\tau,0)=\D\cdot1\ot(\tau,0)$.
\smallskip
 \item\label{item:standard filt prop:P tau zero} The following is a standard filtration of $\fP(\tau,0)$
 \begin{align*}
 \fM(e,\rho)[2]\,\subset\,\D\cdot\left(1\ot(\sigma,-)[1]\right)\,\subset\,\fP(\tau,0).
 \end{align*}
 \qed
 \end{enumerate}
\end{prop}

\subsection{The induced modules}
Given $\lambda\in\Lambda$, we set
\begin{align}\label{eq:fInd}
\fInd(\lambda)=\D\ot_{\D(G)}\lambda\simeq\BV(V)\ot\BV(\oV)\ot\lambda\simeq\fM\bigl(\chgr\BV(\oV)\cdot\lambda\bigr),
\end{align}
where the isomorphisms are of $\Z$-graded $\D^{\leq0}$-modules \cite[Definition 2]{vay-proj}. Thanks to \cite[Theorem 21]{vay-proj} the induced modules help to describe the product in $R^\bullet_{proj}$.

By \cite[(33)]{vay-proj}, $\fInd(\mu)\simeq\oplus_{\lambda\in\Lambda}\,\overline{p_{\fL(\lambda),\mu}}\cdot\fP(\lambda)$. Therefore
\begin{align*}
\fInd(\varepsilon)&\simeq\fP(\varepsilon)\oplus\fP(\tau,1)[2]\oplus\fP(\tau,2)[2],\\
\fInd(e,-)&\simeq(1+t^{4})\cdot\fP(e,-)\oplus(t+t^{3})\cdot\fP(\sigma,+)\oplus\fP(\tau,1)[2]\oplus\fP(\tau,2)[2],\\
\fInd(e,\rho)&\simeq\fP(e,\rho)\oplus(t+t^{3})\cdot\fP(\sigma,+)\oplus\fP(\tau,0)[2]\oplus\fP(\tau,1)[2]\oplus\fP(\tau,2)[2],\\
\fInd(\sigma,-)&\simeq(1+t^{2})\cdot\fP(\sigma,-)\,\oplus2\fP(\sigma,+)[2]\oplus(t+t^{3})\cdot\fP(\tau,1)\oplus(t+t^{3})\cdot\fP(\tau,2),\\
\fInd(\sigma,+)&\simeq(t+t^{3})\cdot\fP(e,-)\,\oplus\fP(e,\rho)[1]\oplus(1+2t^{2}+t^{4})\cdot\fP(\sigma,+)\,\oplus\\
&\qquad\qquad\qquad\qquad\oplus\fP(\tau,0)[1]\oplus(t+t^{3})\cdot\fP(\tau,1)\oplus(t+t^{3})\cdot\fP(\tau,2),\\
\fInd(\tau,0)&\simeq\fP(e,\rho)[2]\oplus(t+t^{3})\cdot\fP(\sigma,+)\oplus\fP(\tau,0)\oplus\fP(\tau,1)[2]\oplus\fP(\tau,2)[2],\\
\fInd(\tau,i)&\simeq\fP(e,-)[2]\oplus\fP(\sigma,-)[1]\oplus(t+t^{3})\cdot\fP(\sigma,+)\oplus\fP(\tau,j)[2]
\end{align*}
for $\{i,j\}=\{1,2\}$.

\subsection{The tensor products of projective modules}

For $\lambda_1,\lambda_2\in\Lambda$, it holds that
\begin{align*}
\fP(\lambda_1)\ot\fP(\lambda_2)\simeq\oplus_{\lambda,\mu\in\Lambda}\quad p_{\fP(\lambda_1),\fW(\lambda)}\, p_{\fP(\lambda_1),\fM(\mu)}\quad\fInd(\lambda\cdot\mu),
\end{align*}
by \cite[Theorem 21]{vay-proj}. The polynomials $p_{\fP(\lambda_1),\fM(\mu)}$ were given at the begining of this section and 
$p_{\fP(\lambda_1),\fW(\lambda)}=t^{-4}p_{\fP(\lambda_1),\fM(\lambda)}$, recall 
\eqref{eq:polinomios de verma y coverma}. The products of weights are in \cite[\S 2.5.4]{PV2}. Thus, the tensor products of the projective modules follow by long and 
tedious computations. For instance,
\begin{align*}
\fP(\varepsilon)&\ot\fP(\varepsilon)\simeq\\
\simeq&t^{-4}(t^8+t^6+4t^4+t^2+1)\fP(\varepsilon)\,\oplus\,2t^{-1}(1+t^2)^2\,\fP(e,-)\,\oplus\,t^{-2}(1+t^2)^3\,\fP(e,\rho)\\
&\,\oplus\,2t^{-3}(1+t^2+t^4)(1+t^2)^2\,\fP(\sigma,-)\,\oplus\,8t^{-1}(1+t^2+t^4)(1+t^2)\,\fP(\sigma,+)\\
&\,\oplus\,t^{-2}(1+t^2)^3\,\fP(\tau,0)\,\oplus\,t^{-2}\left((1+t^4)^2+2(1+t^2)^4\right)\left(\fP(\tau,1)\,\oplus\,\fP(\tau,2)\right).
\end{align*}

\

In the case of the simple projective modules, their fusion rules follow directly since $\fL(\lambda)\simeq\fP(\lambda)\simeq\fM(\lambda)\simeq\fW(\lambda)$.

\begin{prop}\label{prop:simple projective}
Let $\lambda,\mu\in\Lsp$. Hence $\fL(\lambda)\ot\fL(\mu)\simeq\fInd(\lambda\cdot\mu)$.\qed
\end{prop}

%

\subsection{Simple tensoring by projective modules}

To conclude our work, we need to analyze the products $\fL(\lambda)\ot\fL(\mu)$ with $\lambda\in\Lsp$ and $\mu\notin\Lsp$. In this case $\fL(\lambda)$ is projective and hence so are
these tensor products. Thus, we can use the graded character to obtain the following isomorphisms thanks to \eqref{eq:proj D mod}.

\begin{prop}\label{prop:simple by projective}
Let $\{i,j\}=\{1,2\}$. The next isomorphisms hold in the category of graded modules.
\begin{align*}
\fL(e,-)\ot\fL(e,\rho)&\simeq t^{-2}\fP(\tau,0)\\
\fL(e,-)\ot\fL(\tau,0)&\simeq t^{-2}\fP(e,\rho)\\
\fL(e,-)\ot\fL(\sigma,-)&\simeq t^{-1}\bigl(\fL(\tau,1)\oplus\fL(\tau,2)\bigr)\oplus(1+t^{-2})\fL(\sigma,+)
\end{align*}
\begin{align*}
\fL(\tau,i)\ot\fL(e,\rho)&\simeq(1+t^{-2})\fL(\tau,j)\oplus t^{-1}\fL(\sigma,+)\oplus t^{-2}\fP(e,\rho)\\
\fL(\tau,i)\ot\fL(\tau,0)&\simeq(1+t^{-2})\fL(\tau,j)\oplus t^{-1}\fL(\sigma,+)\oplus t^{-2}\fP(\tau,0)\\
\fL(\tau,i)\ot\fL(\sigma,-)&\simeq t^{-1}\bigl(\fL(e,-)\oplus\fL(\tau,j)\bigr)\oplus(1+t^{-2})\fL(\sigma,+)\oplus t^{-2}\fL(\sigma,-)
\end{align*}
\begin{align*}
\fL(\sigma,+)\ot\fL(\tau,0)&\simeq\fL(\sigma,+)\ot\fL(e,\rho)\simeq\\
&\simeq t^{-1}\bigl(\fL(\tau,1)\oplus\fL(\tau,2)\bigr)\oplus(1+t^{-2})\fL(\sigma,+)\oplus t^{-2}\fP(\sigma,-)
\end{align*} 
\begin{align*}
\fL(\sigma,+)&\ot\fL(\sigma,-)\simeq\\
\simeq&(1+t^{-2})\bigl(\fL(e,-)\oplus\fL(\tau,1)\oplus\fL(\tau,2)\bigr)\oplus2t^{-1}\fL(\sigma,+)\oplus t^{-2}\bigl(\fP(e,\rho)\oplus\fP(\tau,0)\bigr)
\end{align*}
\qed
\end{prop}

\section*{Appendix}

We give here the action of the generators of $\D$ on the simple modules $\fL(\lambda)$ for $\lambda\notin\Lsp$. We have computed them identifying $\fL(\lambda)$ with the socle of a Verma module. Then, we use \cite[Appendix A]{PV2} to calculate the action of $\yij{12}$ and the action of $\xij{12}$ is just the multiplication in $\BV(V)$. The actions of the remainder $\yij{ij}$ and $\xij{ij}$ were deduced from the above using that the action is a morphism of $\DSn$-modules. For instance, $\yij{23} c_2=(13)(\yij{12}c_1)$.

\subsection*{The structure of the weights}

The simple Yetter-Drinfeld modules over finite group $G$ were classified for instance in \cite{MR1714540}. The category of Yetter-Drinfeld module is equivalent to the category of modules over the Drinfeld double $G$. Therefore the simple $\D(G)$-module are classified and constructed as follows. Let $\cO_g$ be the conjugacy class of $g\in G$ and $(U,\varrho)$ an irreducible representation of the centralizer $C_g$ of $g$. The corresponding simple $\D(G)$-module is the induced $G$-module $M(g,\varrho)=\ku G\ot_{\ku C_g}U$ with $\ku^G$-action given by $f\cdot(x\ot u)=f(xgx^{-1})x\ot u$ for all function $f\in\ku^G$, $x\in G$ and $u\in U$. Notice that the $\ku^G$-action is equivalent to give a $G$-grading.

In the case of $G=\Sn_3$, we explicitly describe the weights keeping the notation of \cite[\S 2.5.2]{PV2}. Recall also Table \ref{tab:weight}.

\subsubsection*{The weights $(\sigma,\pm)$} The symbols $\mm{12}_{\pm}$, $\mm{23}_{\pm}$ and $\mm{13}_{\pm}$ form a basis. The $\Sn_3$-degree of $\mm{ij}_{\pm}$ is $(ij)$. The $\Sn_3$-action is $g\cdot\mm{ij}_{+}=\mm{g(i)g(j)}_{+}$ and $g\cdot\mm{ij}_{-}=\sgn(g)\mm{g(i)g(j)}_{-}$, respectively; where we identify $\mm{ij}_{\pm}=\mm{ji}_{\pm}$.

\subsubsection*{The weights $(\tau,\ell)$, $\ell=0,1,2$} The symbols $\mm{123}_{\ell}$ and $\mm{132}_{\ell}$ form a basis. The $\Sn_3$-degree of $\mm{ijk}_{\ell}$ is $(ijk)$. Given $g\in G$, we can write $g=(12)^s(123)^t$. Thus, the $\Sn_3$-action is $g\cdot\mm{123}_{\ell}=\zeta^{t\ell}\mm{g(1)g(2)g(3)}_\ell$ and $g\cdot\mm{132}_{\ell}=\zeta^{-t\ell}\mm{g(1)g(3)g(2)}_\ell$; where we identify $\mm{ijk}_\ell=\mm{jki}_\ell=\mm{kij}_\ell$.

\subsubsection*{The weights $\e=(e,+)$ and $(e,-)$} The symbol $\mm{e}_{\pm}$ forms a basis of $\Sn_3$-degree $e$. The $\Sn_3$-action is given by the counit $\e$ and the sign representation of $\Sn_3$, respectively.

\subsubsection*{The weight $(e,\rho)$} The symbols $\mm{123}_{\rho}$ and $\mm{132}_{\rho}$ form a basis. The $\Sn_3$-degree of $\mm{ijk}_{\rho}$ is $e$. As $\Sn_3$-module, it is isomorphic to $(\tau,1)$ via the assignment $\mm{ijk}_\rho\mapsto\mm{ijk}_1$.

\subsection*{Bases for the simple modules}

The isomorphisms listed below are of graded $\DSn$-modules. These are obtained by identifying the elements of the respective ordered bases. 

\

$\bullet$ $\fL(\tau,0)$ has a homogeneous basis $\{a_i\mid1\leq i\leq7\}$ such that
\begin{align*}
\ku\langle a_1,a_2\rangle&\simeq\ku\{\mm{123}_{\rho}, \mm{132}_{\rho}\}\simeq(e,\rho),\quad\deg a_1=\deg a_2=-2,\\
\ku\langle a_3,a_4,a_5\rangle&\simeq\ku\{\mm{12}_+,\mm{13}_+,\mm{23}_+\}\simeq(\sigma,+),\quad\deg a_3=\deg a_4=\deg a_5=-1,\\
\ku\langle a_6,a_7\rangle&\simeq\ku\{\mm{123}_0, \mm{132}_0\}\simeq(\tau,0),\quad\deg a_6=\deg a_7=0,
\end{align*}
The first weight corresponds to $\fC$ in \cite[\S 4.5]{PV2} and the last one to $\BV^{n_{top}}(V)\ot(e,\rho)$.

\

$\bullet$ $\fL(e,\rho)$ has a homogeneous basis $\{b_i\mid1\leq i\leq7\}$ such that
\begin{align*}
\ku\langle b_1,b_2\rangle&\simeq\ku\{\mm{123}_0, \mm{132}_0\}\simeq(\tau,0),\quad\deg b_1=\deg b_2=-2,\\
\ku\langle b_3,b_4,b_5\rangle&\simeq\ku\{\mm{23}_+,\mm{12}_+,\mm{13}_+\}\simeq(\sigma,+),\quad\deg b_3=\deg b_4=\deg b_5=-1,\\
\ku\langle b_6,b_7\rangle&\simeq\ku\{\mm{132}_{\rho},\mm{123}_{\rho}\}\simeq(e,\rho),\quad\deg b_6=\deg b_7=0,
\end{align*}
The first weight corresponds to $\fG$ in \cite[\S 4.6]{PV2} and the last one to $\BV^{n_{top}}(V)\ot(\tau,0)$.

\

$\bullet$ $\fL(\sigma,-)$ has a homogeneous basis $\{c_i\mid1\leq i\leq10\}$ such that
\begin{align*}
\ku\langle c_1,c_2,c_3\rangle&\simeq\ku\{\mm{12}_-,\mm{23}_-,\mm{13}_-\}\simeq(\sigma,-),\quad\deg c_1=\deg c_2=\deg c_3=-2,\\
\ku\langle c_4,c_5\rangle&\simeq\ku\{\mm{123}_1, \mm{132}_1\}\simeq(\tau,1),\quad\deg c_4=\deg c_5=-1,\\
\ku\langle c_6,c_7\rangle&\simeq\ku\{\mm{123}_2, \mm{132}_2\}\simeq(\tau,2),\quad\deg c_6=\deg c_7=-1,\\
\ku\langle c_8,c_9,c_{10}\rangle&\simeq\ku\{\mm{12}_-,\mm{23}_-,\mm{13}_-\}\simeq(\sigma,-),\quad\deg c_8=\deg c_9=\deg c_{10}=0,
\end{align*}
The listed weights correspond to $\BV^{n_{top}}(V)\ot(\sigma,-)$, $\fN_1$, $\fN_2$ and $\fR$ of \cite[\S 4.3]{PV2}, respectively.

\

$\bullet$ $\fL(\varepsilon)=\ku\langle d_1\rangle$ is one-dimensional of degree $0$.

\subsection*{Action on the bases}

We explicitly describe the action of the elements $(ij)$, $\xij{ij}$ and $\yij{ij}$ over the bases above.

\

\noindent
\begin{minipage}[l]{4 cm}
$(12)a_1=a_2$

$(12)a_2=a_1$

$(12)a_3=a_3$

$(12)a_4=a_5$

$(12)a_5=a_4$

$(12)a_6=a_7$

$(12)a_7=a_6$
\end{minipage}
\noindent
\begin{minipage}[c]{4 cm}
$(13)a_1=\zeta^2 a_2$

$(13)a_2=\zeta a_1$

$(13)a_3=a_4$

$(13)a_4=a_3$

$(13)a_5=a_5$

$(13)a_6=a_7$

$(13)a_7=a_6$
\end{minipage}
\noindent
\begin{minipage}[r]{4 cm}
$(23)a_1=\zeta a_2$

$(23)a_2=\zeta^2 a_1$

$(23)a_3=a_3$

$(23)a_4=a_5$

$(23)a_5=a_4$

$(23)a_6=a_7$

$(23)a_7=a_6$
\end{minipage}

\vspace{.3cm}

\noindent
\begin{minipage}[l]{4 cm}
$x_{(12)}a_1=0$

$x_{(12)}a_2=0$

$x_{(12)}a_3=a_1-a_2$

$x_{(12)}a_4=0$

$x_{(12)}a_5=0$

$x_{(12)}a_6=a_5$

$x_{(12)}a_7=-a_4$
\end{minipage}
\noindent
\begin{minipage}[c]{4 cm}
$x_{(13)}a_1=0$

$x_{(13)}a_2=0$

$x_{(13)}a_3=\zeta^2 a_1-\zeta a_2$

$x_{(13)}a_4=0$

$x_{(13)}a_5=0$

$x_{(13)}a_6=a_5$

$x_{(13)}a_7=-a_4$
\end{minipage}
\noindent
\begin{minipage}[r]{4 cm}
$x_{(23)}a_1=0$

$x_{(23)}a_2=0$

$x_{(23)}a_3=0$

$x_{(23)}a_4=\zeta a_1-\zeta^2 a_2$

$x_{(23)}a_5=0$

$x_{(23)}a_6=a_3$

$x_{(23)}a_7=-a_5$
\end{minipage}

\vspace{.3cm}

\noindent
\begin{minipage}[l]{4 cm}
$y_{(12)}a_1=a_3$

$y_{(12)}a_2=-a_3$

$y_{(12)}a_3=0$

$y_{(12)}a_4=-a_7$

$y_{(12)}a_5=a_6$

$y_{(12)}a_6=0$

$y_{(12)}a_7=0$
\end{minipage}
\noindent
\begin{minipage}[c]{4 cm}
$y_{(13)}a_1=\zeta a_3$

$y_{(13)}a_2=-\zeta^2 a_3$

$y_{(13)}a_3=0$

$y_{(13)}a_4=-a_7$

$y_{(13)}a_5=a_6$

$y_{(13)}a_6=0$

$y_{(13)}a_7=0$
\end{minipage}
\noindent
\begin{minipage}[r]{4 cm}
$y_{(23)}a_1=\zeta^2 a_4$

$y_{(23)}a_2=-\zeta a_4$

$y_{(23)}a_3=a_6$

$y_{(23)}a_4=0$

$y_{(23)}a_5=-a_7$

$y_{(23)}a_6=0$

$y_{(23)}a_7=0$
\end{minipage}

\vspace{.5cm}

\noindent
\begin{minipage}[l]{4 cm}
$(12)b_1=b_2$

$(12)b_2=b_1$

$(12)b_3=b_5$

$(12)b_4=b_4$

$(12)b_5=b_3$

$(12)b_6=b_7$

$(12)b_7=b_6$
\end{minipage}
\noindent
\begin{minipage}[c]{4 cm}
$(13)b_1=b_2$

$(13)b_2=b_1$

$(13)b_3=b_4$

$(13)b_4=b_3$

$(13)b_5=b_5$

$(13)b_6=\zeta b_7$

$(13)b_7=\zeta^2 b_6$
\end{minipage}
\noindent
\begin{minipage}[r]{4 cm}
$(23)b_1=b_2$

$(23)b_2=b_1$

$(23)b_3=b_3$

$(23)b_4=b_5$

$(23)b_5=b_4$

$(23)b_6=\zeta^2 b_7$

$(23)b_7=\zeta b_6$
\end{minipage}

\vspace{.3cm}

\noindent
\begin{minipage}[l]{4 cm}
$x_{(12)}b_1=0$

$x_{(12)}b_2=0$

$x_{(12)}b_3=b_1$

$x_{(12)}b_4=0$

$x_{(12)}b_5=-b_2$

$x_{(12)}b_6=b_4$

$x_{(12)}b_7=-b_4$
\end{minipage}
\noindent
\begin{minipage}[c]{4 cm}
$x_{(13)}b_1=0$

$x_{(13)}b_2=0$

$x_{(13)}b_3=-b_2$

$x_{(13)}b_4=b_1$

$x_{(13)}b_5=0$

$x_{(13)}b_6=\zeta^2 b_5$

$x_{(13)}b_7=-\zeta b_5$
\end{minipage}
\noindent
\begin{minipage}[r]{4 cm}
$x_{(23)}b_1=0$

$x_{(23)}b_2=0$

$x_{(23)}b_3=0$

$x_{(23)}b_4=-b_2$

$x_{(23)}b_5=b_1$

$x_{(23)}b_6=\zeta b_3$

$x_{(23)}b_7=-\zeta^2 b_3$
\end{minipage}

\vspace{.3cm}

\noindent
\begin{minipage}[l]{4 cm}
$y_{(12)}b_1=b_3$

$y_{(12)}b_2=-b_5$

$y_{(12)}b_3=0$

$y_{(12)}b_4=b_6-b_7$

$y_{(12)}b_5=0$

$y_{(12)}b_6=0$

$y_{(12)}b_7=0$
\end{minipage}
\noindent
\begin{minipage}[c]{4 cm}
$y_{(13)}b_1=b_4$

$y_{(13)}b_2=-b_3$

$y_{(13)}b_3=0$

$y_{(13)}b_4=0$

$y_{(13)}b_5=\zeta b_6 - \zeta^2 b_7$

$y_{(13)}b_6=0$

$y_{(13)}b_7=0$
\end{minipage}
\noindent
\begin{minipage}[r]{4 cm}
$y_{(23)}b_1=b_5$

$y_{(23)}b_2=-b_4$

$y_{(23)}b_3=\zeta^2 b_6-\zeta b_7$

$y_{(23)}b_4=0$

$y_{(23)}b_5=0$

$y_{(23)}b_6=0$

$y_{(23)}b_7=0$
\end{minipage}

\vspace{.5cm}

\noindent
\begin{minipage}[l]{4 cm}
$(12)c_1=-c_1$

$(12)c_2=-c_3$

$(12)c_3=-c_2$

$(12)c_4=c_5$

$(12)c_5=c_4$

$(12)c_6=c_7$

$(12)c_7=c_6$

$(12)c_8=-c_8$

$(12)c_9=-c_{10}$

$(12)c_{10}=-c_9$
\end{minipage}
\noindent
\begin{minipage}[c]{4 cm}
$(13)c_1=-c_2$

$(13)c_2=-c_1$

$(13)c_3=-c_3$

$(13)c_4=\zeta^2 c_5$

$(13)c_5=\zeta c_4$

$(13)c_6=\zeta c_7$

$(13)c_7=\zeta^2 c_6$

$(13)c_8=-c_9$

$(13)c_9=-c_8$

$(13)c_{10}=-c_{10}$
\end{minipage}
\noindent
\begin{minipage}[r]{4 cm}
$(23)c_1=-c_3$

$(23)c_2=-c_2$

$(23)c_3=-c_1$

$(23)c_4=\zeta c_5$

$(23)c_5=\zeta^2 c_4$

$(23)c_6=\zeta^2 c_7$

$(23)c_7=\zeta c_6$

$(23)c_8=-c_{10}$

$(23)c_9=-c_9$

$(23)c_{10}=-c_8$
\end{minipage}

\vspace{.3cm}

\noindent
\begin{minipage}[l]{4 cm}
$x_{(12)}c_1=0$

$x_{(12)}c_2=0$

$x_{(12)}c_3=0$

$x_{(12)}c_4=\zeta c_2$

$x_{(12)}c_5=\zeta c_3$

$x_{(12)}c_6=\zeta^2 c_2$

$x_{(12)}c_7=\zeta^2 c_3$

$x_{(12)}c_8=0$

$x_{(12)}c_9=\frac{1}{1-\zeta}(c_6-\zeta c_4)$

$x_{(12)}c_{10}=\frac{1}{1-\zeta}(c_7-\zeta c_5)$
\end{minipage}
\noindent
\begin{minipage}[c]{4 cm}
$x_{(13)}c_1=0$

$x_{(13)}c_2=0$

$x_{(13)}c_3=0$

$x_{(13)}c_4=\zeta^2 c_1$

$x_{(13)}c_5=c_2$

$x_{(13)}c_6=\zeta c_1$

$x_{(13)}c_7=c_2$

$x_{(13)}c_8=\frac{1}{1-\zeta}(\zeta c_6-c_4)$

$x_{(13)}c_9=\frac{\zeta^2}{1-\zeta}(c_7- c_5)$

$x_{(13)}c_{10}=0$
\end{minipage}
\noindent
\begin{minipage}[r]{4 cm}
$x_{(23)}c_1=0$

$x_{(23)}c_2=0$

$x_{(23)}c_3=0$

$x_{(23)}c_4=c_3$

$x_{(23)}c_5=\zeta^2 c_1$

$x_{(23)}c_6=c_3$

$x_{(23)}c_7=\zeta c_1$

$x_{(23)}c_8=\frac{1}{1-\zeta}(\zeta c_7-c_5)$

$x_{(23)}c_9=0$

$x_{(23)}c_{10}=\frac{\zeta^2}{1-\zeta}(c_6- c_4)$
\end{minipage}

\vspace{.3cm}

\noindent
\begin{minipage}[l]{4 cm}
$y_{(12)}c_1=0$

$y_{(12)}c_2=\frac{\zeta^2}{1-\zeta}(c_4-\zeta c_6)$

$y_{(12)}c_3=\frac{\zeta^2}{1-\zeta}(c_5-\zeta c_7)$

$y_{(12)}c_4=c_9$

$y_{(12)}c_5=c_{10}$

$y_{(12)}c_6=c_9$

$y_{(12)}c_7=c_{10}$

$y_{(12)}c_8=0$

$y_{(12)}c_9=0$

$y_{(12)}c_{10}=0$
\end{minipage}
\noindent
\begin{minipage}[c]{4 cm}
$y_{(13)}c_1=\frac{1}{1-\zeta}(\zeta c_4-c_6)$

$y_{(13)}c_2=\frac{1}{1-\zeta}(c_5-\zeta c_7)$

$y_{(13)}c_3=0$

$y_{(13)}c_4=\zeta c_8$

$y_{(13)}c_5=\zeta^2 c_9$

$y_{(13)}c_6=\zeta^2 c_8$

$y_{(13)}c_7=\zeta c_9$

$y_{(13)}c_8=0$

$y_{(13)}c_9=0$

$y_{(13)}c_{10}=0$
\end{minipage}
\noindent
\begin{minipage}[r]{4 cm}
$y_{(23)}c_1=\frac{1}{1-\zeta}(\zeta c_5-c_7)$

$y_{(23)}c_2=0$

$y_{(23)}c_3=\frac{1}{1-\zeta}(c_4-\zeta c_6)$

$y_{(23)}c_4=\zeta^2 c_{10}$

$y_{(23)}c_5=\zeta c_8$

$y_{(23)}c_6=\zeta c_{10}$

$y_{(23)}c_7=\zeta^2 c_8$

$y_{(23)}c_8=0$

$y_{(23)}c_9=0$

$y_{(23)}c_{10}=0$
\end{minipage}

\bibliography{refs}{}

\def\cprime{$'$}
\begin{thebibliography}{10}

\bibitem{arXiv:1802.00316}
N.~Andruskiewitsch and I.~Angiono.
\newblock {O}n {N}ichols algebras over basic {H}opf algebras.
\newblock {\em arXiv:1802.00316}, 2016.

\bibitem{MR1714540}
N.~Andruskiewitsch and M.~Gra{\~n}a.
\newblock Braided {H}opf algebras over non-abelian finite groups.
\newblock {\em Bol. Acad. Nac. Cienc. (C{\'o}rdoba)}, 63:45--78, 1999.
\newblock Colloquium on Operator Algebras and Quantum Groups (Spanish)
  (Vaquer{\'{\i}}as, 1997).

\bibitem{dim72}
N.~Andruskiewitsch and C.~Vay.
\newblock On a family of hopf algebras of dimension 72.
\newblock {\em Bull. Belg. Math. Soc. Simon Stevin}, 19(3):415--443, 2012.

\bibitem{spherical1}
J.~W. Barrett and B.~W. Westbury.
\newblock Spherical categories.
\newblock {\em Adv. Math.}, 143:357--375, 1999.

\bibitem{arXiv:1705.08024}
G.~Bellamy and U.~Thiel.
\newblock Highest weight theory for finite-dimensional graded algebras with
  triangular decomposition.
\newblock {\em Adv. Math.}, 330:361 -- 419, 2018.

\bibitem{MR3260906}
H.-X. Chen.
\newblock The {G}reen ring of {D}rinfeld double {$D(H_4)$}.
\newblock {\em Algebr. Represent. Theory}, 17(5):1457--1483, 2014.

\bibitem{MR3655701}
H.-X. Chen, H.~S.~E. Mohammed, and H.~Sun.
\newblock Indecomposable decomposition of tensor products of modules over
  {D}rinfeld doubles of {T}aft algebras.
\newblock {\em J. Pure Appl. Algebra}, 221(11):2752--2790, 2017.

\bibitem{MR3148512}
H.-X. Chen, F.~Van~Oystaeyen, and Y.~Zhang.
\newblock The {G}reen rings of {T}aft algebras.
\newblock {\em Proc. Amer. Math. Soc.}, 142(3):765--775, 2014.

\bibitem{MR1243707}
C.~Cibils.
\newblock A quiver quantum group.
\newblock {\em Comm. Math. Phys.}, 157(3):459--477, 1993.

\bibitem{MR2184820}
K.~Erdmann, E.~L. Green, N.~Snashall, and R.~Taillefer.
\newblock Representation theory of the {D}rinfeld doubles of a family of {H}opf
  algebras.
\newblock {\em J. Pure Appl. Algebra}, 204(2):413--454, 2006.

\bibitem{MR1667680}
S.~Fomin and A.~N. Kirillov.
\newblock Quadratic algebras, {D}unkl elements, and {S}chubert calculus.
\newblock In {\em Advances in geometry}, volume 172 of {\em Progr. Math.},
  pages 147--182. Birkh\"auser Boston, Boston, MA, 1999.

\bibitem{GaiSemTip06}
A.~M. Gainutdinov, A.~M. Semikhatov, I.~Yu. Tipunin, and B.~L. Feigin.
\newblock {K}azhdan--{L}usztig correspondence for the~representation category
  of the~triplet $w$-algebra in logarithmic {CFT}.
\newblock {\em TMF}, 148(3):398--427, 2006.

\bibitem{MR2681259}
A.~Garc\'ia~Iglesias.
\newblock Representations of finite dimensional pointed {H}opf algebras over
  {$\mathbb{S}_3$}.
\newblock {\em Rev. Un. Mat. Argentina}, 51(1):51--77, 2010.

\bibitem{MR659212}
R.~Gordon and E.~L. Green.
\newblock Graded {A}rtin algebras.
\newblock {\em J. Algebra}, 76(1):111--137, 1982.

\bibitem{MR1987337}
E.~Gunnlaugsd\'ottir.
\newblock Monoidal structure of the category of {$u^+_q$}-modules.
\newblock {\em Linear Algebra Appl.}, 365:183--199, 2003.
\newblock Special issue on linear algebra methods in representation theory.

\bibitem{MR2774620}
H.~Kondo and Y.~Saito.
\newblock Indecomposable decomposition of tensor products of modules over the
  restricted quantum universal enveloping algebra associated to
  {${\mathfrak{sl}}_2$}.
\newblock {\em J. Algebra}, 330:103--129, 2011.

\bibitem{MR2279242}
L.~Krop and D.~E. Radford.
\newblock Simple modules for the {D}rinfel\cprime d double of a class of {H}opf
  algebras.
\newblock In {\em Groups, rings and algebras}, volume 420 of {\em Contemp.
  Math.}, pages 229--235. Amer. Math. Soc., Providence, RI, 2006.

\bibitem{LENTNER2017264}
S.~Lentner and J.~Priel.
\newblock {A} decomposition of the {B}rauer-{P}icard group of the
  representation category of a finite group.
\newblock {\em Journal of Algebra}, 489:264 -- 309, 2017.

\bibitem{MR3077243}
L.~Li and Y.~Zhang.
\newblock The {G}reen rings of the generalized {T}aft {H}opf algebras.
\newblock In {\em Hopf algebras and tensor categories}, volume 585 of {\em
  Contemp. Math.}, pages 275--288. Amer. Math. Soc., Providence, RI, 2013.

\bibitem{MR1435369}
M.~Lorenz.
\newblock Representations of finite-dimensional {H}opf algebras.
\newblock {\em J. Algebra}, 188(2):476--505, 1997.

\bibitem{MR1800714}
A.~Milinski and H.-J. Schneider.
\newblock Pointed indecomposable {H}opf algebras over {C}oxeter groups.
\newblock In {\em New trends in {H}opf algebra theory ({L}a {F}alda, 1999)},
  volume 267 of {\em Contemp. Math.}, pages 215--236. Amer. Math. Soc.,
  Providence, RI, 2000.

\bibitem{NIKSHYCH2014191}
D.~Nikshych and B.~Riepel.
\newblock {C}ategorical {L}agrangian {G}rassmannians and {B}rauer-{P}icard
  groups of pointed fusion categories.
\newblock {\em Journal of Algebra}, 411:191 -- 214, 2014.

\bibitem{MR2046303}
C.~N\u{a}st\u{a}sescu and F.~Van~Oystaeyen.
\newblock {\em Methods of graded rings}, volume 1836 of {\em Lecture Notes in
  Mathematics}.
\newblock Springer-Verlag, Berlin, 2004.

\bibitem{PV2}
B.~Pogorelsky and C.~Vay.
\newblock Verma and simple modules for quantum groups at non-abelian groups.
\newblock {\em Adv. Math.}, 301:423--457, 2016.

\bibitem{MR1284788}
R.~Suter.
\newblock Modules over {$\mathfrak{U}_q(\mathfrak{sl}_2)$}.
\newblock {\em Comm. Math. Phys.}, 163(2):359--393, 1994.

\bibitem{vay-proj}
C.~{Vay}.
\newblock On projective modules over finite quantum groups.
\newblock {\em Transf. Groups}, 2017.

\bibitem{vay-mca17}
C.~{Vay}.
\newblock On {H}opf algebras with triangular decomposition.
\newblock {\em arXiv:1808.03799}, 2018.

\bibitem{MR1969453}
M.~Wakui.
\newblock On representation rings of non-semisimple {H}opf algebras of low
  dimension.
\newblock In {\em Proceedings of the 35th {S}ymposium on {R}ing {T}heory and
  {R}epresentation {T}heory ({O}kayama, 2002)}, pages 9--14. Symp. Ring Theory
  Represent. Theory Organ. Comm., Okayama, 2003.

\bibitem{MR1367852}
S.~J. Witherspoon.
\newblock The representation ring of the quantum double of a finite group.
\newblock {\em J. Algebra}, 179(1):305--329, 1996.

\bibitem{MR2441478}
Y.~Zhang, F.~Wu, L.~Liu, and H.-X. Chen.
\newblock Grothendieck groups of a class of quantum doubles.
\newblock {\em Algebra Colloq.}, 15(3):431--448, 2008.

\end{thebibliography}
\bibliographystyle{plain}
\end{document}